\documentclass{article}

\usepackage{mathptmx}
\usepackage{amssymb,amsmath,amsthm,amsfonts}
\usepackage{graphicx}
\usepackage{bm}
\usepackage[all,cmtip]{xy}
\usepackage{tikz-cd}
\usepackage{stmaryrd}
\usepackage{xcolor}
\usepackage{mathrsfs}
\usepackage[shortlabels]{enumitem}

\usepackage{MnSymbol}
\definecolor{linkblue}{HTML}{003d73}
\definecolor{linkgreen}{HTML}{006161}
\definecolor{linkred}{HTML}{a11950}
\usepackage[pagebackref]{hyperref}
\hypersetup{
	pdftitle={Fusion Frame Homotopy and Tightening Fusion Frames by Gradient Descent},
	pdfauthor={Tom Needham and Clayton Shonkwiler},
	pdfsubject={frame theory},
	pdfkeywords={fusion frames, frame homotopy, operator-valued frames, fusion frame potential, Benedetto-Fickus},
	colorlinks=true,
	linkcolor=linkblue,
	citecolor=linkgreen,
	urlcolor=linkred
}
\usepackage{algorithm}
\usepackage{algpseudocode}
\usepackage{pifont}
\usepackage[margin=1in]{geometry}
\usepackage{mathtools}
\usepackage[percent]{overpic}
\usepackage{booktabs}
\usepackage{authblk}
\usepackage{cite}
\usepackage[nameinlink]{cleveref}
\usepackage{quiver}
\usepackage{mathdots}

\usepackage{enumitem}

\crefname{thm}{theorem}{theorems}
\crefname{cor}{corollary}{corollaries}
\crefname{question}{question}{questions}
\crefname{prop}{proposition}{propositions}

\newtheorem{thm}{Theorem}[section]
\newtheorem*{thm*}{Theorem}
\newtheorem{prop}[thm]{Proposition}
\newtheorem{lem}[thm]{Lemma}
\newtheorem{cor}[thm]{Corollary}
\newtheorem{conj}[thm]{Conjecture} 

\theoremstyle{definition}
\newtheorem{definition}[thm]{Definition}
\newtheorem{example}[thm]{Example}
\newtheorem{remark}[thm]{Remark}
\newtheorem{question}[thm]{Question}

\newcommand{\R}{\mathbb{R}}
\newcommand{\Z}{\mathbb{Z}}

\newcommand{\mathh}{\mathbb{H}}
\newcommand{\C}{\mathbb{C}}

\newcommand{\K}{\mathbb{K}}

\newcommand{\Gr}[2]{\operatorname{Gr}_{#1}(\C^{#2})}

\newcommand{\Id}{\mathbb{I}}
\newcommand{\tr}{\operatorname{tr}}

\newcommand{\unitary}{\operatorname{U}}
\newcommand{\hermitian}{\mathcal{H}}

\newcommand{\Ad}{\operatorname{Ad}}
\renewcommand{\Im}{\operatorname{Im}}

\newcommand{\ad}{\operatorname{ad}}
\newcommand{\oframes}{\mathcal{OF}^{\K^d,\bk}}
\newcommand{\coframes}{\mathcal{OF}^{\C^d,\bk}}
\newcommand{\roframes}{\mathcal{OF}^{\R^d,\bk}}
\newcommand{\pframes}{\mathcal{PF}^{\C^d,\bk}}
\newcommand{\fframes}{\mathcal{FF}^{\K^d,\bk}}
\newcommand{\cfframes}{\mathcal{FF}^{\C^d,\bk}}
\newcommand{\rfframes}{\mathcal{FF}^{\R^d,\bk}}
\newcommand{\gtorus}{\operatorname{U}(k_1) \times \dots \times \operatorname{U}(k_N)}
\newcommand{\spec}{\operatorname{spec}}
\newcommand{\FFP}{\operatorname{FFP}}
\newcommand{\OFP}{\FFP}
\newcommand{\EFP}{\operatorname{EFP}}
\newcommand{\mbf}[1]{\boldsymbol{#1}}
\newcommand{\grad}{\operatorname{grad}}
\newcommand{\rk}{\operatorname{rk}}

\newcommand{\SL}{\operatorname{SL}}
\newcommand{\GL}{\operatorname{GL}}

\newcommand{\spa}{\operatorname{span}}
\newcommand{\row}{\operatorname{row}}

\newcommand{\br}{\boldsymbol{r}}
\newcommand{\bk}{\boldsymbol{k}}
\newcommand{\blam}{\boldsymbol{\lambda}}
\newcommand{\bmu}{\boldsymbol{\mu}}
\newcommand{\bLam}{\boldsymbol{\Lambda}}

\makeatletter
\newcommand{\subalign}[1]{%
  \vcenter{%
    \Let@ \restore@math@cr \default@tag
    \baselineskip\fontdimen10 \scriptfont\tw@
    \advance\baselineskip\fontdimen12 \scriptfont\tw@
    \lineskip\thr@@\fontdimen8 \scriptfont\thr@@
    \lineskiplimit\lineskip
    \ialign{\hfil$\m@th\scriptstyle##$&$\m@th\scriptstyle{}##$\hfil\crcr
      #1\crcr
    }%
  }%
}
\makeatother

\title{Fusion Frame Homotopy and Tightening Fusion Frames by Gradient Descent}
\author[$\ast$]{Tom Needham}
\author[$\dag$]{Clayton Shonkwiler}
\affil[$\ast$]{Department of Mathematics, Florida State University, Tallahassee, FL} 
\affil[$\dag$]{Department of Mathematics, Colorado State University, Fort Collins, CO}
\date{}

\begin{document}

\maketitle

\begin{abstract}
Finite frames, or spanning sets for finite-dimensional Hilbert spaces, are a ubiquitous tool in signal processing. There has been much recent work on understanding the global structure of collections of finite frames with prescribed properties, such as spaces of unit norm tight frames. We extend some of these results to the more general setting of fusion frames---a fusion frame is a collection of subspaces of a finite-dimensional Hilbert space with the property that any vector can be recovered from its list of projections. The notion of tightness extends to fusion frames, and we consider the following basic question: is the collection of tight fusion frames with prescribed subspace dimensions path connected? We answer (a generalization of) this question in the affirmative, extending the analogous result for unit norm tight frames proved by Cahill, Mixon and Strawn. We also extend a result of Benedetto and Fickus, who defined a natural functional on the space of unit norm frames (the frame potential), showed that its global minimizers are tight, and showed that it has no spurious local minimizers, meaning that gradient descent can be used to construct unit-norm tight frames. We prove the analogous result for the fusion frame potential of Casazza and Fickus, implying that, when tight fusion frames exist for a given choice of dimensions, they can be constructed via gradient descent. Our proofs use techniques from symplectic geometry and Mumford's geometric invariant theory.
\end{abstract}

\section{Introduction}

A \emph{frame} is a spanning collection of vectors in a Hilbert space which satisfies a certain Parseval-like condition~\cite{Duffin:1952kh}. Frames are important in the context of signal processing, where a signal is modeled as a vector in a Hilbert space and is encoded by the list of its inner products with the vectors in the frame~\cite{Daubechies:1986kk}. A frame is generally \emph{overcomplete} (i.e., linearly dependent), a property which is useful in applications because the resulting signal representations are, by virtue of their redundancy, more robust to noise than representations in a basis. For practical reasons, there has been increased interest in finite frames, that is, spanning sets of vectors for finite-dimensional Hilbert spaces; see  \cite{Casazza:2013ft} and \cite{Waldron:2018wc} for general surveys. For the rest of the paper, we will only consider finite frames and will therefore work under the simplifying convention that our Hilbert space is $\K^d$, $\K = \R$ or $\C$, endowed with the standard inner product $\langle \cdot, \cdot \rangle$, with standard norm denoted $\|\cdot \|$.

Typically, one requires a frame to satisfy certain norm or spectral constraints. For example, a frame $\{f_1,\ldots,f_N\}$ for $\K^d$ is called \emph{tight} if the \emph{frame operator} map  $v \mapsto \sum_{i=1}^N \langle v, f_i \rangle f_i$ is a constant multiple of the identity map on $\K^d$---tight frames are of particular interest, since they guarantee optimal robustness under certain noise models~\cite{Casazza:2003vp, Holmes:2004iv}. Equivalently, a frame is tight if the spectrum of its frame operator is constant. The collection of all length-$N$ frames for $\K^d$ with prescribed frame vector norms and frame operator spectrum defines a complicated algebraic variety, and topological and geometrical properties of these varieties have been the focus of a growing body of recent research \cite{dykema2003manifold,Cahill:2013fy,cahill2017connectivity,Needham:2021bi,needham_toric_2021,needham2022admissibility}. 

The goal of this paper is to extend results on spaces of frames to the setting of a more general signal processing tool: a \emph{fusion frame} is an ordered collection $(\mathcal{S}_1,\ldots,\mathcal{S}_N)$ of subspaces of $\K^d$ such that the \emph{frame operator}
$
	v \mapsto \sum_{i=1}^N P_i v
$
is invertible (and hence necessarily positive definite), where $P_i:\K^d \to \mathcal{S}_i$ is orthogonal projection. Observe that if all subspaces are 1-dimensional, this essentially reduces to the definition of a (classical) frame. Fusion frames were introduced by Casazza and Kutyniok in \cite{casazza_frames_2004} as a hierarchical approach to the construction of large frames with desirable properties. Fusion frames were subsequently developed into a general  tool for distributed data processing \cite{Casazza:2008ku,casazza_modeling_2007,kutyniok_robust_2009}---the basic idea is that factors such as hardware limitations may require a collection of local vector-valued signal measurements to be coherently and robustly fused into a global measurement. Fusion frames have more recently been applied to compressed sensing for structured signals or low quality  measurement modalities~\cite{ayaz2016uniform,xia2017nonuniform,aceska2020local}.

As in the case of (classical) frames, there is typically a focus on fusion frames with extra structure. For instance, the notion of a tight frame generalizes to that of a \emph{tight fusion frame}---this is a fusion frame such that the frame operator is a multiple of the identity map. We study the topology and geometry of spaces of tight fusion frames, as well as spaces of fusion frames with more general prescribed data. We give precise formulations of our results in subsections \ref{sub:fusion frame homotopy} and \ref{sub:optimization} below, but our main contributions are described informally as follows.

\begin{itemize}
\item Fusion Frame Homotopy Theorem (\Cref{thm:main}).

We show that the space of tight fusion frames in a complex Hilbert space with a prescribed sequence of subspace dimensions is path connected. This gives a generalization of the (complex) Frame Homotopy Theorem, which says that the space of length-$N$ tight frames for a complex Hilbert space, whose frame vectors are all unit norm, is path connected. This was proved by Cahill, Mixon and Strawn in~\cite{cahill2017connectivity}, affirming a conjecture of Larson from 2002 (although the conjecture was first published in~\cite{dykema2003manifold}). The Frame Homotopy Theorem was generalized by the authors of the present paper to spaces of frames with more general prescribed spectral and norm data using techniques from symplectic geometry~\cite{Needham:2021bi} and to spaces of quaternionic frames using the theory of isoparametric submanifolds~\cite{Palais:1988ks}. We once again apply symplectic techniques to prove the analogous result for fusion frames---in fact, our techniques work in much greater generality, and we are able to prove a connectivity result for spaces of \emph{operator-valued frames}. Our result is described in detail below in \Cref{sub:fusion frame homotopy}.

\item Benedetto--Fickus Theorem for Fusion Frames (\Cref{thm:descent}).

In~\cite{benedetto2003finite}, Benedetto and Fickus introduced the \emph{frame potential}, a natural energy functional on the space of frames consisting of a fixed number of unit vectors in a fixed Hilbert space. They showed that the global minimizers of the frame potential are tight frames and proved the surprising result that the frame potential has no spurious local minimizers, meaning that tight frames can reliably be generated via gradient descent---we refer to this result (\!\!\cite[Theorem 7.1]{benedetto2003finite}, also stated below as \Cref{thm:benedetto fickus}) as the \emph{Benedetto--Fickus theorem}. Casazza and Fickus defined a more general functional on the space of fusion frames, called the \emph{fusion frame potential}, and characterized its minimizers~\cite{casazza_minimizing_2009}. We extend the Benedetto--Fickus theorem to give general conditions which guarantee that the fusion frame potential has no spurious local minimizers (\Cref{thm:descent}), which implies that if tight fusion frames exist in a given space of fusion frames, they can always be reached by gradient descent (\Cref{cor:local minima are global}). Together with Mixon and Villar, we gave several strengthenings of the Benedetto--Fickus theorem in~\cite{mixon_three_2021}, one of which was proved using ideas from Mumford's \emph{geometric invariant theory} (GIT)~\cite{mumford_geometric_1994}. \Cref{thm:descent} is proved by extending this application of GIT to the fusion frame setting. We precisely state and further contextualize our result below in \Cref{sub:optimization}.

\end{itemize}

The structure of the paper is follows. The remaining subsections of the introduction pin down exact definitions, set notation, and give precise statements of our main results. \Cref{sec:symplectic_machinery} provides necessary background on symplectic geometry, in preparation for the proof of the Fusion Frame Homotopy Theorem (\Cref{thm:main}), which is then proved in \Cref{sec:sg frames}. The Benedetto--Fickus Theorem for Fusion Frames (\Cref{thm:benedetto fickus}) is proved in \Cref{sec:benedetto fickus}, after introducing the main ideas of GIT. We remark that the exposition about symplectic geometry and GIT in \Cref{sec:symplectic_machinery,sec:benedetto fickus}, respectively, is intended to be accessible to non-experts in these fields. The paper concludes with a discussion of open problems and future directions in \Cref{sec:discussion}.

\subsection{Fusion Frame Homotopy} 
\label{sub:fusion frame homotopy}

Recall that a \emph{tight fusion frame} (TFF) is a fusion frame $(\mathcal{S}_1,\ldots,\mathcal{S}_N)$ such that the frame operator $\sum_i P_i$ is a multiple of the identity. If $k_i = \dim(\mathcal{S}_i) = \rk(P_i)$ then, since all nonzero eigenvalues of an orthogonal projector are equal to 1, 
\[
	\tr\left(\sum_{i=1}^N P_i\right) = \sum_{i=1}^N \tr(P_i) = \sum_{i=1}^N k_i =:n,
\]
so it must be the case that a TFF has frame operator equal to $\frac{n}{d} \Id_d$, where $\Id_d$ is the identity operator on $\K^d$. Since the $P_i$ uniquely determine and are uniquely determined by the $\mathcal{S}_i$, we will also call $(P_1, \dots , P_N)$ a (tight) fusion frame when the corresponding $(\mathcal{S}_1, \dots , \mathcal{S}_N)$ is.

With the frame homotopy conjecture~\cite{dykema2003manifold} (resolved by Cahill, Mixon, and Strawn in 2017~\cite{cahill2017connectivity}) in mind, the following is a very natural question:

\begin{question}\label{q:fusion frame homotopy}
	Is the space of TFFs in $\K^d$ with given ranks $(k_1, \dots, k_N)$ path-connected?
\end{question}

The first goal of this paper is to show that the answer to \Cref{q:fusion frame homotopy} is always ``yes'' for complex fusion frames (i.e., when $\K = \C$). In fact, we will prove a much more general theorem about spaces of operator-valued frames with fixed spectral data. 

To motivate the definition given below, notice that, for a fusion frame $(P_1, \dots , P_N)$, each projector $P_i$ has a full-rank square root $A_i: \K^d \to \K^{k_i}$ so that $A_i^\ast A_i = P_i$. This square root is unique up to composing with a unitary transformation of the codomain:

\begin{prop}[{\!\!\cite[Theorem~7.3.11]{horn_matrix_2013}}]\label{prop:indeterminacy}
	Suppose $A: \K^d \to \K^k$ and $B: \K^d \to \K^k$ are linear maps. Then $A^\ast A = B^\ast B$ if and only if there exists a unitary transformation $U \in \unitary(k)$ so that $B = UA$. 
\end{prop}

Hence, up to this indeterminacy, we can also think of a fusion frame as a collection of operators $(A_1, \dots , A_N)$ so that $A_i^\ast A_i$ is an orthogonal projector for each $i=1,\dots , N$ and so that $A_1^\ast A_1 + \dots + A_N^\ast A_N$ is positive definite. This is the definition we will generalize by relaxing the condition on the individual $A_i^\ast A_i$:

\begin{definition}\label{def:operator-valued frames}
	Let $d, N, k_1,\dots ,k_N$ be positive integers. Let $\bk :=(k_1, \dots , k_N)$. Then an \emph{operator-valued frame} of type $(d, \bk)$ is an $N$-tuple $\mbf{A}:=(A_1, \dots, A_N)$ of linear maps $A_i:\K^d \to \K^{k_i}$ so that the \emph{frame operator}
	\[
		S_{\mbf{A}} := \sum_{i=1}^N A_i^\ast A_i
	\]
	 is positive definite. The space of operator-valued frames of type $(d, \bk)$ will be denoted $\oframes$.
\end{definition}

This definition is essentially the specialization to finite dimensions of Kaftal, Larson, and Zhang's definition~\cite{Kaftal:2009ek} of an operator-valued frame. As with fusion frames, we will define $P_i := A_i^\ast A_i$, so that the frame operator is $\sum_i P_i$. In practice, we will make the simplifying assumption $\rk(A_i) = k_i$, by restricting the codomain of $A_i$ to its image if necessary.

The feature that distinguishes fusion frames among the operator-valued frames is that the $P_i$ are orthogonal projectors of rank $k_i$, which means precisely that the $k_i$ nonzero eigenvalues of $P_i$ are all equal to 1. More generally, we can consider spaces of operator-valued frames with fixed spectral data:

\begin{definition}\label{def:fixed spectra}
	Let $d, N, k_1, \dots , k_N$ be positive integers and $\bk:=(k_1, \dots , k_N)$. For each $i$, let $\br_i =(r_{i1}, \dots , r_{ik_i})$ with $r_{i1} \geq \dots \geq r_{ik_i} > 0$,\footnote{Remember our simplifying assumption that the $A_i$ should be full rank.} and let $\br = (\br_1, \dots , \br_N)$. $\oframes(\br)$ will denote the space of operator-valued frames $(A_1, \dots , A_N)$ of type $(d,\bk)$ for which $P_i = A_i^\ast A_i$ has nonzero eigenvalues equal to $\br_i$. Equivalently, each $r_{ij} = \sigma_{ij}^2$, where the $\sigma_{ij}$ are the singular values of $A_i$.
\end{definition}

\begin{example}\label{ex:frames}
    Let $\bk = (k_1,\ldots,k_N)$ and let $\br = (\br_1,\ldots,\br_N)$ such that every $\br_i$ is a list of $k_i$ ones. Then $\mathcal{OF}^{\K^d,\bk}(\br)$ is equivalent to the space of fusion frames with prescribed ranks $\bk$. Indeed, it is the space of $\mbf{A} = (A_1, \dots , A_N)$ so that $P_i = A_i^\ast A_i$ is a rank-$k_i$ orthogonal projector and the frame operator $S_{\mbf{A}} = \sum_i P_i$ is positive-definite. Since this space is of particular interest, we use the specialized notation $\fframes := \mathcal{OF}^{\K^d,\bk}(\br)$. As a special case, if $k_i = 1$ for all $i$ then $\fframes$ is equivalent to the space of unit-norm frames of length $N$ in $\K^d$.
\end{example}

Tight fusion frames are also distinguished among fusion frames by spectral data since multiples of the identity are uniquely determined by their spectra: $\lambda \Id_d$ is the only operator with spectrum $(\lambda, \dots , \lambda)$. Fixing the spectral data of the frame operator is more natural than fixing the frame operator itself, since we can always diagonalize the frame operator by precomposing the $A_i: \K^d \to \K^{k_i}$ by a common unitary transformation of the domain. Hence, a natural generalization of tight frames is the collection of operator-valued frames with fixed spectrum of their frame operator:

\begin{definition}\label{def:fixed frame operator spectrum}
	Let $d, N, k_1, \dots , k_N$ be positive integers  and $\bk:=(k_1, \dots , k_N)$. Let $\blam = (\lambda_1, \dots , \lambda_d)$ with $\lambda_1 \geq \dots \geq \lambda_d > 0$. $\oframes_{\blam}$ will denote the space of operator-valued frames $(A_1, \dots, A_N)$ whose frame operator $S_{\mbf{A}} = \sum_i A_i^\ast A_i$ has spectrum $\blam$.
\end{definition}

\begin{example}
    Let $\bk = (k_1,\ldots,k_N)$ and let $\blam = (1,\ldots,1)$ be the list of $d$ ones.  Then $\mathcal{OF}^{\K^d,\bk}_{\blam}$ could reasonably be called the space of Parseval operator-valued frames of type $(d, \bk)$, by analogy with the case $k_i = 1$ for all $i$, when $\mathcal{OF}^{\K^d,\bk}_{\blam}$ is equivalent to the space of Parseval frames of length $N$ in $\K^d$ (that is, frames whose frame operator is the identity).
\end{example}

Of course, the definition of a tight fusion frame includes both fixed spectral data of the $P_i$ and fixed spectral data of the frame operator, so it involves intersecting two of the spaces defined above. In that spirit, define
\[
	\oframes_{\blam}(\br):= \oframes_{\blam} \cap \oframes(\br).
\]

Hence, the operator-valued generalization of \Cref{q:fusion frame homotopy} is the following:

\begin{question}\label{q:operator-valued frame homotopy}
	For given $d,N,\bk, \br, \blam$, is the space $\oframes_{\blam}(\br)$ path-connected?
\end{question}

In the case of complex classical frames (i.e., $\K = \C$ and $\bk = (1, \dots , 1)$, but general $\blam$ and $\br$) we answered this question in the affirmative using symplectic geometry~\cite{Needham:2021bi}. Here, we give an affirmative answer for general complex operator-valued frames:

\begin{thm}\label{thm:main}
	The space $\coframes_{\blam}(\br)$ is always path-connected.
\end{thm}

Since the empty set is trivially path-connected, the substantive content of this theorem is that $\coframes_{\blam}(\br)$ is path-connected whenever it is non-empty. For discussion of when it is empty, see \Cref{sec:admissibility}.

\subsection{Benedetto--Fickus Theorem for Fusion Frames}
\label{sub:optimization}

We now specialize back to fusion frames, but again work over $\K = \R$ or $\C$. Recall from \Cref{ex:frames} that, for $d$ and $N$ positive and $\bk = (k_1, \dots , k_N)$, $\fframes$ denotes the space of (square roots of) fusion frames with prescribed ranks $\bk$.

It is easy to generate random elements of $\fframes$: for each $i$, let $B_i$ be a $k_i \times d$ matrix with standard Gaussian entries, and let $A_i$ be the result of orthogonalizing the rows of $B_i$. Given their usefulness in applications, it is desirable to generate not just fusion frames, but \emph{tight} fusion frames. Following the lead of Benedetto and Fickus~\cite{benedetto2003finite}, a plausible strategy for doing so is to define a potential function on $\fframes$ whose global minima are exactly the set of TFFs, and then flow along the negative gradient direction of this potential. A natural candidate for such a potential is the fusion frame potential defined by Casazza and Fickus~\cite{casazza_minimizing_2009}, generalizing Benedetto and Fickus' frame potential. In the definition below, and at times throughout the rest of the paper, we will abuse notation and also use $\langle \cdot, \cdot \rangle$ and $\|\cdot\|$ to denote, respectively, the Frobenius inner product and norm on a space of operators; the meaning should always be clear from context.

\begin{definition}\label{def:fusion frame potential}
	Let $d, N, k_1, \dots , k_N$ be positive integers. The \emph{fusion frame potential} $\FFP:\fframes \to \R$ is defined by
	\[
		\FFP(\mbf{A}) := \left\|S_{\mbf{A}}\right\|^2 .
	\]	
\end{definition}

Note that the fusion frame potential could be generalized to arbitrary spaces of operator-valued frames, though we will not do so here. As Casazza and Fickus showed (see also \Cref{prop:welch bound}), the fusion frame potential satisfies a Welch-type lower bound which is achieved exactly at the TFFs. Hence, when they exist, TFFs are exactly the global minima of $\FFP$, and it is natural to ask whether we can get to these global minima by negative gradient flow:

\begin{question}\label{q:local minima} 
	When they exist, can we construct TFFs by flowing along the negative gradient of $\FFP$? That is, are all local minima of $\FFP$ also global in this setting?
\end{question}

An affirmative answer to this question is essentially a conjecture of Massey, Ruiz, and Stojanoff~\cite[Conjecture~4.3.3]{Massey:2009iu}. For unit-norm tight frames, Benedetto and Fickus showed that there are no spurious local minima of the frame potential, completely resolving \Cref{q:local minima} in this case:

\begin{thm}[{Benedetto--Fickus theorem~\cite{benedetto2003finite}}]\label{thm:benedetto fickus}
	Let $d$ and $N$ be positive and $\bk=(1,\dots , 1)$. Then $\FFP: \fframes \to \R$ has no spurious local minimizers.
\end{thm}

In~\cite{mixon_three_2021} we gave three new proofs of this result, one of which we now intend to generalize to fusion frames. The exposition here is self-contained, but parallels that in~\cite[\S4]{mixon_three_2021}.

To state our theorem, we need to define a suitable notion of genericity for fusion frames.

\begin{definition}\label{def:property S}
	Let $\mbf{A} \in \fframes$ and recall that, for each $i=1,\dots , N$, the image of the orthogonal projector $P_i = A_i^\ast A_i$ is a $k_i$-dimensional subspace $\mathcal{S}_i \subset \K^d$. We say that $\mbf{A}$ \emph{has property} $\mathscr{S}$ if, for all proper linear subspaces $\mathcal{Q} \subset \K^d$, 
	\[
		\frac{1}{\dim \mathcal{Q}} \sum_{i=1}^N \dim(\mathcal{S}_i \cap \mathcal{Q}) \leq \frac{1}{d} \sum_{i=1}^N k_i = \frac{n}{d}.
	\]
\end{definition}

Roughly speaking, this condition says that no subspace of $\K^d$ intersects too many of the $\mathcal{S}_i$. For example, in the classical frames case $\bk = (1, \dots , 1)$, property $\mathscr{S}$ says the fraction of frame vectors lying in any given $r$-dimensional subspace is no more than $\frac{r}{d}$. In particular, this is a much weaker condition than being full spark.

More generally, in the case of equal-rank fusion frames---i.e., each $P_i$ is rank $k$---property $\mathscr{S}$ is weaker than the condition $\sum_{i=1}^N \dim(\mathcal{S}_i \cap \mathcal{Q}) \leq \dim \mathcal{Q}$ for all proper subspaces $\mathcal{Q} \subset \K^d$. Fusion frames satisfying this latter condition are relevant to the problem of compressed sensing with block sparsity~\cite{Eldar:2010gf,Boufounos:2009kk} since they provide unique reconstructions for the largest possible class of block-sparse signals, much as classical full spark frames are optimal for traditional compressed sensing~\cite{Alexeev:2012jk,donoho_optimally_2003}.

\begin{thm}\label{thm:descent}
	Let $0< d,N,k_1,\dots,k_N$. Consider the negative gradient flow $\Gamma: \fframes \times [0,\infty) \to \fframes$ defined by the differential equation
	\[
		\Gamma(\mbf{A}_0,0) = \mbf{A}_0, \quad \frac{d}{dt}\Gamma(\mbf{A}_0,t) = - \grad \FFP(\Gamma(\mbf{A}_0,t))
	\]
	where $\grad$ is the Riemannian gradient and $\mbf{A}_0 = (A_1, \dots , A_n) \in \fframes$ is an initial frame.
	
	If $\mbf{A}_0$ has property $\mathscr{S}$, then $\mbf{A}_\infty := \lim_{t \to \infty} \Gamma(\mbf{A}_0,t) \in \fframes$ is a global minimum of $\FFP$.
\end{thm}

As mentioned above, for parameters $d$ and $\bk$ which admit TFFs, the TFFs are exactly the global minima of $\FFP$, so for those parameters this theorem implies that fusion frames with property $\mathscr{S}$ always flow to tight fusion frames. Moreover, these are exactly the parameters for which $\fframes$ contains fusion frames with property $\mathscr{S}$ (\Cref{cor:tight iff S}), which turn out to be dense in $\fframes$ (\Cref{cor:dense}). This implies that there are fusion frames arbitrarily close to any non-minimal critical point of $\FFP$ which flow to a global minimum, so $\FFP$ cannot have any local minima that are not global minima.

\begin{cor}\label{cor:local minima are global}
	When $\fframes$ contains TFFs, all local minima of $\FFP$ are global minima.
\end{cor}

This generalizes Benedetto and Fickus' result to fusion frames and completely answers \Cref{q:local minima}. See also work of Heineken, Llarena, and Morillas~\cite{heineken_minimizers_2018}, which gives a similar answer for a restriction of $\FFP$ to a subset of $\fframes$, and of Massey, Ruiz, and Stojanoff~\cite{Massey:2009iu}, who proved an analogous result with respect to a different notion of distance on fusion frames.

Even in some situations where there cannot be any TFFs---for example, when $d = N = 3$ and $\bk=(1,1,2)$---the negative gradient flow of $\FFP$ empirically seems to avoid spurious local minima, so there is hope that the conclusion of \Cref{cor:local minima are global} follows from weaker hypotheses.

\begin{remark}[Admissibility]
\label{sec:admissibility}

In light of \Cref{cor:local minima are global}, it would be useful to know when the space of TFFs is non-empty. More generally, we can ask whether $\oframes_{\blam}(\br)$ is non-empty, which is really prior to the question of whether it is path-connected. In the context of classical frames (i.e., $\bk = (1, \dots , 1)$), this is sometimes called the \emph{admissibility problem}, and its resolution follows easily from the Schur--Horn theorem~\cite{schur1923uber,horn_matrix_2013} (see also~\cite{antezana_schur--horn_2007,casazza2010} and Cahill, Fickus, Mixon, Poteet, and Strawn's contructive solution~\cite{Cahill:2013jv}): in this setting, $\br = (r_1, \dots , r_N)$ is a single list of positive numbers, and $\oframes_{\blam}(\br)$ is nonempty precisely when $\blam$ \emph{majorizes} $\br$~\cite{marshall1979inequalities}, meaning that 
\[
	\sum_{i=1}^d \lambda_i = \sum_{i=1}^N r_i \qquad \text{and} \qquad \sum_{i=1}^k \lambda_i \geq \sum_{i=1}^k r_i \quad \text{for all } k=1,\dots , d,
\]
where we assume $\blam$ and $\br$ are sorted in decreasing order. In particular, the space $\oframes_{\left(\frac{N}{d}, \dots , \frac{N}{d}\right)}(1, \dots , 1)$ of unit-norm tight frames is always nonempty when $N \geq d$. 

More generally, admissibility has been completely resolved for all spaces of tight fusion frames (therefore precisely determining the scope of \Cref{thm:descent} and \Cref{cor:local minima are global}): by Casazza, Fickus, Mixon, Wang, and Zhou~\cite{casazza_constructing_2011} when all $P_i$ have the same rank, and in general by Bownik, Luoto, and Richmond~\cite{bownik_combinatorial_2015}. Even for tight fusion frames, the conditions on $d$ and $\bk$ which ensure non-emptiness are quite complicated, boiling down to non-vanishing of certain Littlewood--Richardson coefficients. 

However, this is not really surprising, given that in general admissibility for operator-valued frames is equivalent to the following question: what $\blam, \br_1, \dots , \br_N$ can be the eigenvalues of $d \times d$ Hermitian matrices $M, P_1, \dots , P_N$ so that $M=P_1 + \dots + P_N$? In 1962, Horn conjectured necessary and sufficient conditions involving a complicated system of inequalities between $\blam$ and the $\br_i$~\cite{horn_eigenvalues_1962}. This became known as the Horn conjecture, which was eventually proved in the late 1990s by Klyachko~\cite{klyachko_stable_1998} and Knutson and Tao~\cite{knutson_honeycomb_1999,knutson_honeycomb_2004}; see Fulton's survey~\cite{fulton_eigenvalues_2000} for more, and Berenstein and Sjamaar's paper~\cite{berenstein_coadjoint_2000} for a generalization. The answer does not depend on the base field: if $\coframes_{\blam}(\br)$ is non-empty, then so is $\roframes_{\blam}(\br)$ (see, e.g.,~\cite[Theorem~20]{fulton_eigenvalues_2000}).

Thus, the admissibility problem for operator-valued frames is at least implicitly resolved by the proof of the Horn conjecture (see also~\cite{Massey:2009iu}), and more explicit solutions exist in the most common cases of interest, namely frames and tight fusion frames. 
\end{remark}

\section{Symplectic Machinery} 
\label{sec:symplectic_machinery}

As in our earlier work on frames~\cite{Needham:2021bi}, our strategy for proving \Cref{thm:main} involves symplectic geometry. First, we will review some standard concepts from symplectic geometry, which will help both to provide a quick reference and to establish our notation and sign conventions. Our main references for symplectic geometry are McDuff and Salamon~\cite{mcduff2017introduction} and Cannas da Silva~\cite{CannasdaSilva:2001cg}.

\subsection{Definitions} 
\label{sub:definitions}

A \emph{symplectic manifold} is a pair $(M,\omega)$, where $M$ is a (smooth, real) manifold and $\omega$ is a closed, nondegenerate 2-form on $M$. For each point $p \in M$ and for each pair of tangent vectors $X,Y \in T_pM$, we write $\omega_p(X,Y) \in \R$ for the evaluation of $\omega$ at the point on the pair. Being closed means that $d\omega$ is identically zero, where $d$ is the exterior derivative on $M$, and being nondegenerate means that, for each $p \in M$ and each $X \in T_pM$, there exists $Y \in T_pM$ so that $\omega_p(X,Y) \neq 0$. Nondegeneracy implies that a symplectic manifold must be even-dimensional over $\R$.

\begin{example}
	The simplest example of a symplectic manifold is $\C^n \approx \R^{2n}$. For $p \in \C^n$ there is a natural isomorphism $T_p \C^n \approx \C^n$, and complex coordinates $(x_1 + \sqrt{-1} y_1, \dots , x_n + \sqrt{-1}y_n)$ for $\C^n$ correspond to real coordinates $(x_1, \dots , x_n, y_1, \dots , y_n)$, with respect to which the \emph{standard symplectic form} on $\C^n$ is given by
	\[
		\omega = dx_1 \wedge dy_1 + \dots + dx_n \wedge dy_n.
	\]
	We can rewrite this in complex coordinates as follows: given $p \in \C^n$ and $Z=(z_1, \dots , z_n), W = (w_1, \dots , w_n) \in T_p\C^n \approx \C^n$, 
	\[
		\omega_p(Z,W) = -\Im(\overline{w}_1 z_1 + \dots + \overline{w}_n z_n) = -\Im(W^\ast Z) = -\Im\langle Z, W \rangle.
	\]
\end{example}
In fact, Darboux's theorem~\cite[Theorem~3.2.2]{mcduff2017introduction} implies that every point in a symplectic manifold has a neighborhood on which the symplectic structure looks like the standard one on $\C^n$.

Symplectic geometry grew out of Hamiltonian mechanics, and a big emphasis both historically and currently is on the interactions between symplectic structures and symmetries; that is, group actions on symplectic manifolds. In general, if $G$ is a Lie group with Lie algebra $\mathfrak{g}$ and $G$ acts on a manifold $M$, then each $\xi \in \mathfrak{g}$ determines a vector field $Y_\xi$ on $M$ as follows. For $p \in M$ and $g \in G$, let $g \cdot p$ denote the action of $g$ on $p$. Then we define
\[
	\left. Y_\xi \right|_p := \left.\frac{d}{d\varepsilon} \right|_{\varepsilon=0} \exp(\varepsilon \xi) \cdot p,
\]
where $\exp: \mathfrak{g} \to F$ is the exponential map of $G$. 

Now, suppose $G$ acts on a symplectic manifold $(M,\omega)$. If $\mathfrak{g}^\ast$ is the dual of $\mathfrak{g}$, a map $\Phi: M \to \mathfrak{g}^\ast$ is called a \emph{momentum map} for the $G$-action if it satisfies the following conditions.

First, let $D\Phi(p): T_pM \to T_{\Phi(p)}\mathfrak{g}^\ast$ be the derivative of $\Phi$ at $p \in M$. Since $\mathfrak{g}^\ast$ is a vector space, there is a natural isomorphism $T_{\Phi(p)}\mathfrak{g}^\ast \approx \mathfrak{g}^\ast$, so we can interpret $D\Phi(p)$ as a map to $\mathfrak{g}^\ast$. Hence, for each $X \in T_p M$, $D\Phi(p)(X)$ is an element of the dual space of $\mathfrak{g}$; that is, $D\Phi(p)(X): \mathfrak{g} \to \R$. For $\Phi$ to be a momentum map we require this map to satisfy the compatibility condition
\[
	D\Phi(p)(X)(\xi) = \omega_p(\left.Y_\xi\right|_p,X)
\]
for all $\xi \in \mathfrak{g}$.

Also, we require a momentum map $\Phi$ to be equivariant with respect to the given action of $G$ on $M$ and the natural coadjoint action of $G$ on $\mathfrak{g}^\ast$. More explicitly, the \emph{adjoint action} of $G$ on $\mathfrak{g}$ is the linearization at the identity of the conjugation action of $G$ on itself; that is, for each $g \in G$ the map $\Ad_g: \mathfrak{g} \to \mathfrak{g}$ is the derivative at the identity of the map $h \mapsto g h g^{-1}$. In turn, the \emph{coadjoint action} of $G$ on $\mathfrak{g}^\ast$ gives a map $\Ad_g^\ast: \mathfrak{g}^\ast \to \mathfrak{g}^\ast$ for each $g \in G$ which is defined by $\Ad_g^\ast(\chi)(\xi) := \chi(\Ad_{g^{-1}}(\xi))$. When $G$ is a matrix group, both $\Ad_g$ and $\Ad_g^\ast$ can be interpreted as conjugation by $g$. Now, the equivariance condition on momentum maps is that, for each $g \in G$ and each $p \in M$,
\[
	\Ad_g^\ast(\Phi(p)) = \Phi(g \cdot p).
\]

When a $G$-action admits a momentum map, the action is called \emph{Hamiltonian} and the tuple $(M, \omega, G, \Phi)$ is a \emph{Hamiltonian $G$-space}.

\begin{prop}[{cf.~\cite[Exercise~5.3.16]{mcduff2017introduction}}]\label{prop:products}
	Suppose $G$ is a Lie group and $(M_i,\omega_i,G,\Phi_i)$ are Hamiltonian $G$-spaces for $i=1,\dots,n$. Then the diagonal action of $G$ on $M_1 \times \dots \times M_n$ is Hamiltonian with momentum map
	\[
		\Phi(p_1, \dots , p_n) = \Phi_1(p_1) + \dots + \Phi_n(p_n).
	\]
\end{prop}

\begin{proof}
	The standard symplectic form on a product $\prod_{i=1}^n M_i$ of symplectic manifolds is $\pi_1^\ast \omega_1 + \dots + \pi_n^\ast \omega_n$, where ${\pi_k:\prod_{i=1}^n M_i\to M_k}$ is projection onto the $k$th factor and $\pi_k^\ast$ is the induced pullback operator on forms.
	
	Any tangent vector $X \in T_{(p_1, \dots , p_n)} \prod_{i=1}^n M_i$ is of the form $X = (X_1, \dots , X_n)$ for $X_i \in T_{p_i}M_i$ for all $i=1, \dots , n$. In particular, if $\xi \in \mathfrak{g}$, then the associated vector field 
	\[
		\left.Y_\xi\right|_{(p_1, \dots , p_n)} = (\left. (Y_1)_\xi\right|_{p_1}, \dots , \left. (Y_n)_\xi\right|_{p_n}),
	\]
	where each $\left. (Y_i)_\xi\right|_{p_i} = (d\pi_i)_{(p_1, \dots , p_n)} \left.Y_\xi\right|_{(p_1, \dots , p_n)}$ is the vector field on $M_i$ associated to $\xi$. 
	
	Let $\Phi:\prod_{i=1}^n M_i \to \mathfrak{g}^\ast$ be defined as in the statement of the proposition and let $X = (X_1, \dots , X_n) \in T_{(p_1, \dots , p_n)} \prod_{i=1}^n M_i$. Then
	\begin{align*}
		D\Phi(p_1,\dots,p_n)(X_1,\dots, X_n)(\xi) = \sum_{i=1}^n D\Phi_i(p_i)(\xi) = \sum_{i=1}^n (\omega_i)_{p_i}(\left.Y_\xi\right|_{p_i},X_i) & = \sum_{i=1}^n(\omega_1)_{p_i}((d\pi_i)_{(p_i, \dots , p_i)} \left.Y_\xi\right|_{(p_1, \dots , p_n)},X_i)\\
		& = \sum_{i=1}^n\pi_i^\ast(\omega_i)_{p_i}(\left.Y_\xi\right|_{(p_1, \dots , p_n)},X_i)\\
		& = \omega_{(p_1,\dots,p_n)}(\left.Y_\xi\right|_{(p_1, \dots ,p_n)},(X_1,\dots, X_n)),
	\end{align*}
	where we've used linearity in various place. So $\Phi$ satisfies the appropriate compatibility condition with $\omega$.
	
	Moreover, if $g \in G$ and $(p_1, \dots , p_n) \in \prod_{i=1}^n M_i$, then
	\[
		\Ad_g^\ast(\Phi(p_1, \dots, p_n)) = \Ad_g^\ast\left(\sum_{i=1}^n \Phi_i(p_i)\right) = \sum_{i=1}^n \Ad_g^\ast(\Phi_i(p_i)) = \sum_{i=1}^n \Phi_i(g \cdot p_i) = \Phi(g \cdot p_1, \dots , g \cdot p_n) = \Phi(g \cdot (p_1, \dots , p_n))
	\]
	using linearity of $\Ad_g^\ast$ and $G$-equivariance of the $\Phi_i$, so $\Phi$ is $G$-equivariant, and hence is a momentum map for the $G$-action.
\end{proof}

\subsection{Coadjoint Orbits} 
\label{sub:coadjoint orbits}

An important class of Hamiltonian spaces are coadjoint orbits, which we now describe in some detail, loosely following~\cite[Example~5.3.11]{mcduff2017introduction}.
	
Let $G$ be a Lie group with Lie algebra $\mathfrak{g}$ and dual Lie algebra $\mathfrak{g}^\ast$. Let $\chi \in \mathfrak{g}^\ast$ and let $\mathcal{O}_\chi$ be the \emph{coadjoint orbit} through $\chi$; that is,
\[
	\mathcal{O}_\chi := \left\{\Ad_g^\ast(\chi) \,|\, g \in G \right\}.
\]
It is a standard fact that $\mathcal{O}_\chi$ has a natural symplectic form called the \emph{Kirillov--Kostant--Souriau (KKS) form}~\cite[II.3.c]{audin2012torus}, denoted $\omega^{\text{KKS}}$, defined as follows. The tangent space to $\mathcal{O}_\chi$ at $\chi$ consists of vectors of the form $\ad_\xi^\ast\chi$, where $\xi \in \mathfrak{g}$ and $\ad_\xi^\ast$ is the coadjoint representation of $\mathfrak{g}$ on $\mathfrak{g}^\ast$; that is, the derivative of the coadjoint representation $\Ad^\ast:G \to \operatorname{Aut}(\mathfrak{g}^\ast)$ at the identity. Then
\[
	\omega^{\text{KKS}}_\chi(\ad_\xi^\ast (\chi),\ad_{\zeta}^\ast (\chi)) := \chi([\xi,\zeta]),
\]
where $[\cdot,\cdot]$ is the Lie bracket on $\mathfrak{g}$. 

The action of $G$ on $\mathcal{O}_\chi$ is Hamiltonian, with momentum map $\Phi: \mathcal{O}_\chi \to \mathfrak{g}^\ast$ simply being the inclusion map. To see this, first notice that the vector field $Y_\xi$ on $\mathcal{O}_\chi$ induced by $\xi \in \mathfrak{g}$ is
\[
	\left. Y_\xi \right|_\chi = \left. \frac{d}{d\varepsilon}\right|_{\varepsilon=0} \Ad_{\exp \varepsilon \xi}^\ast(\chi) = \ad_\xi^\ast \chi.
\]
If $\Phi$ is the inclusion of $\mathcal{O}_\chi$ into $\mathfrak{g}^\ast$, then its derivative $D\Phi(\chi): T_\chi\mathcal{O}_\chi \to \mathfrak{g}^\ast$ is also an inclusion, so
\begin{multline*}
	D\Phi(\chi)(\ad_\zeta^\ast (\chi))(\xi) = \ad_\zeta^\ast (\chi)(\xi) = \chi(-\ad_\zeta(\xi)) = \chi(-[\zeta,\xi]) = \chi([\xi,\zeta])  \\
	= \omega^{\text{KKS}}(\ad_\xi^\ast (\chi),\ad_\zeta^\ast (\chi)) = \omega^{\text{KKS}}(\left. Y_\xi\right|_\chi,\ad_\zeta^\ast (\chi)).
\end{multline*}
Since $\Phi$ is obviously equivariant, this shows that it is a momentum map for the $G$-action on $\mathcal{O}_\chi$.

Combining the above discussion with \Cref{prop:products} yields the following corollary:

\begin{cor}\label{cor:products of coadjoint orbits}
	Let $G$ be a Lie group with dual Lie algebra $\mathfrak{g}^\ast$. Let $\mathcal{O}_1, \dots , \mathcal{O}_n \subset \mathfrak{g}^\ast$ be coadjoint orbits with their KKS forms $\omega_i^{\text{KKS}}$. Then $\left(\prod_{i=1}^n\mathcal{O}_i, \sum_{i=1}^n \pi_i^\ast \omega_i^{\text{KKS}}\right)$ is symplectic and the diagonal coadjoint action of $G$ is Hamiltonian with momentum map
	\[
		\Phi(\chi_1, \dots, \chi_n) = \chi_1 + \dots + \chi_n,
	\]
	where on the right hand side we have identified each $\chi_i \in \mathcal{O}_i$ with its inclusion into $\mathfrak{g}^\ast$.
\end{cor}
	
\subsection{Level Sets of Momentum Maps}
\label{sub:level sets}

We are shortly going to associate $P_i$ with fixed spectra with points on a coadjoint orbit of $\unitary(d)$, so fixing the spectra of $P_1, \dots , P_N$ corresponds to choosing coadjoint orbits of $\unitary(d)$ and taking their Cartesian product, and the associated momentum map of the diagonal $\unitary(d)$ action will be the frame operator by \Cref{cor:products of coadjoint orbits}. This will all be explained in the next section, but the point is that the \emph{tight} operator-valued frames with fixed spectra of the $P_i$ will be precisely a level set of the momentum map, and the frame homotopy problem in this case boils down to showing connectedness of this level set. Fortunately for us, there are powerful results showing that level sets of momentum maps are usually connected. 

While one can prove somewhat more general results, we will only consider Hamiltonian $G$-spaces $(M,\omega,G,\Phi)$ where both $M$ and $G$ are compact. Since we intend to apply these results to the diagonal action of $\unitary(d)$ on a product of its coadjoint orbits, this will suffice for our purposes.

Fix a $G$-invariant inner product on the Lie algebra $\mathfrak{g}$. This induces a vector space isomorphism of $\mathfrak{g}^\ast$ with $\mathfrak{g}$, and hence determines an inner product and norm on the codomain $\mathfrak{g}^\ast$ of the momentum map $\Phi$. Kirwan~\cite{Kirwan:1984jt} showed that, while it is not quite a Morse--Bott function, the norm-squared map $\|\Phi\|^2$ has many of the desirable properties of such functions. 

\begin{thm}[{Kirwan~\cite[Theorem~4.16]{Kirwan:1984jt}}]\label{thm:kirwan stratification}
	The set of critical points for $\|\Phi\|^2$ is a finite disjoint union of closed subsets $\{C_\beta: \beta \in \mathcal{B}\}$ on each of which $\|\Phi\|^2$ takes a constant value (here the indexing set $\mathcal{B}$ is a finite subset of the positive Weyl chamber $\mathfrak{t}_+$ of $\mathfrak{g} \simeq \mathfrak{g}^\ast$).
	
	Moreover, there is a smooth stratification $\{S_\beta: \beta \in \mathcal{B}\}$ of $M$, where $p \in S_\beta$ if and only if the limit set of its image under the flow of $-\grad\|\Phi\|^2$ (for an appropriate choice of $G$-invariant Riemannian metric on $M$) is contained in $C_\beta$. 
	
	Finally, for each $\beta$ the inclusion $C_\beta \hookrightarrow S_\beta$ is an equivalence of Cech cohomology and also $G$-invariant cohomology.
\end{thm}

Moreover, Kirwan showed~\cite[4.18]{Kirwan:1984jt} that the $S_\beta$ are all locally closed, even-dimensional submanifolds of $M$. Since it is impossible to disconnect a manifold by removing submanifolds of codimension $\geq 2$, each stratum must be connected. In particular, the stratum $S_0$ of points which flow to $\Phi^{-1}(0)$ is connected; by the equivalence of Cech cohomology, the level set $\Phi^{-1}(0) = C_0$ is also connected:

\begin{thm}[{Kirwan~\cite[(3.1)]{Kirwan:1984im}; see also~\cite[Remark~5.8]{Sjamaar:1991fe}}]\label{thm:kirwan connected}
	Let $(M,\omega,G,\Phi)$ be a Hamiltonian $G$-space with $M$ and $G$ compact. Then $\Phi^{-1}(0)$ is connected.
\end{thm}

In general, we will be interested not just in $\Phi^{-1}(0)$, but also in $\Phi^{-1}(\mathcal{O})$, where $\mathcal{O} \subset \mathfrak{g}^\ast$ is a coadjoint orbit. Fortunately, the ``shifting trick'' will allow us to easily translate \Cref{thm:kirwan connected} to this more general setting:

\begin{cor}\label{cor:general connected}
	Let $(M,\omega,G,\Phi)$ be a Hamiltonian $G$-space with $M$ and $G$ compact and let $\mathcal{O} \subset \mathfrak{g}^\ast$ be a coadjoint orbit. Then $\Phi^{-1}(\mathcal{O})$ is connected.
\end{cor}

\begin{proof}
	The goal is to build a new Hamiltonian $G$-space $(\overline{M},\overline{\omega},G,\overline{\Phi})$ so that $\overline{\Phi}^{-1}(0) \approx \Phi^{-1}(\mathcal{O})$. This is exactly what the shifting trick (see~\cite[\S24.4]{CannasdaSilva:2001cg} or~\cite[Proof of Proposition~5.4.15]{mcduff2017introduction}) was designed to do. 
	
	Specifically, we know that $(\mathcal{O},\omega^{\text{KKS}})$ is symplectic, and hence so is $(\mathcal{O},-\omega^{\text{KKS}})$. Let $\overline{M} = M \times \mathcal{O}$ with symplectic form $\overline{\omega} = \pi_1^\ast \omega + \pi_2^\ast(-\omega^{\text{KKS}})$, where as usual $\pi_i$ is projection onto the $i$th factor. Let the $G$ action on $\overline{M}$ be the diagonal action. Then \Cref{prop:products} and the discussion in \Cref{sub:coadjoint orbits} imply that
	\[
		\overline{\Phi}(p,\chi) := \Phi(p) - \chi
	\]
	is a momentum map for this action.
	
	Now,
	\[
		\overline{\Phi}^{-1}(0) = \{(p,\Phi(p))\,|\,p \in M, \Phi(p) \in \mathcal{O}\}
	\]
	is connected by \Cref{thm:kirwan connected}. Since this space is certainly homeomorphic to $\Phi^{-1}(\mathcal{O})$, we conclude that $\Phi^{-1}(\mathcal{O})$ is also connected.
\end{proof}

\section{The Symplectic Geometry of Spaces of Operator-Valued Frames} 
\label{sec:sg frames}

We now relate the general machinery of the previous section to operator-valued frames. Throughout this section we will fix positive integers $d,N,k_1, \dots , k_N$. We will also fix $\br = (\br_1,\dots, \br_N)$, where $\br_i = (r_{i1},\dots , r_{ik_i})$ with $r_{i1} \geq \dots \geq r_{ik_i} > 0$. 

Given this data, consider the space $\coframes(\br)$ of all operator-valued frames $(A_1, \dots , A_n)$, where $A_i: \C^d \to \C^{k_i}$ is linear and $P_i := A_i^\ast A_i$ has spectrum $\br_i$. The space $\coframes(\br)$ is not obviously symplectic, but our first goal is to show that the quotient $\coframes(\br)/(\gtorus)$ is symplectic, and in fact is essentially a product of coadjoint orbits of $\unitary(d)$.

To start, we recall that the space $\hermitian(d)$ of $d \times d$ Hermitian matrices can be identified with the dual $\mathfrak{u}(d)^\ast$ of the Lie algebra of $\unitary(d)$ by the isomorphism
\begin{align*}
	\alpha: \hermitian(d) & \to \mathfrak{u}(d)^\ast \\
	\xi & \mapsto \left( \eta \mapsto \frac{\sqrt{-1}}{2} \tr(\eta^\ast \xi) = \left\langle\frac{\sqrt{-1}}{2}\xi, \eta \right\rangle =: \alpha_\xi(\eta)\right).
\end{align*}

We collect relevant lemmas from our previous paper~\cite[\S2.2.1 and \S2.2.2]{Needham:2021bi} in the following proposition:

\begin{prop}\label{prop:hermitian coadjoint}
	Under this identification, the coadjoint action of $\unitary(d)$ on $\mathfrak{u}(d)^\ast$ corresponds to the conjugation action of $\unitary(d)$ on $\hermitian(d)$, and hence coadjoint orbits of $\unitary(d)$ can be identified with collections of Hermitian matrices with fixed spectrum $\bmu$, which we will denote $\mathcal{O}_{\bmu} \subset \hermitian(d)$.
\end{prop}

For each $i=1, \dots , N$, the collection of $P_i \in \hermitian(d)$ with spectrum $\br_i$ is exactly $\mathcal{O}_{\br_i}$, and so 
\begin{equation}\label{eq:pframes1}
	\pframes(\br) := \{(P_1, \dots , P_N) \in \hermitian(d)^N \,|\, \spec(P_1) = \br_1,\dots , \spec(P_N) = \br_N\} = \mathcal{O}_{\br_1} \times \dots \times \mathcal{O}_{\br_N}
\end{equation}
is a product of coadjoint orbits. \Cref{cor:products of coadjoint orbits} now has the following immediate consequence:
\begin{cor}\label{cor:frame operator coadjoint orbits}
	The momentum map $\Phi$ for the diagonal conjugation action of $\unitary(d)$ on $\pframes(\br)$ is precisely the frame operator:
	\[
		\Phi(P_1, \dots , P_N) = P_1 + \dots + P_N.
	\]
\end{cor}

Fix $\blam = (\lambda_1, \dots , \lambda_d)$ with $\lambda_1 \geq \dots \geq \lambda_d > 0$. Define
\begin{equation}\label{eq:pframes2}
	\pframes_{\blam}(\br) := \{(P_1, \dots , P_N) \in \hermitian(d)^N \,|\, \spec(P_1) = \br_1,\dots , \spec(P_N) = \br_N, \spec(P_1 + \dots + P_N) = \blam\}.
\end{equation}
In other words, this is the collection of $P_i$ with both the right individual spectra and the right spectrum of the frame operator. Then \Cref{cor:frame operator coadjoint orbits,cor:general connected} imply:

\begin{prop}\label{prop:connected operators}
	$\pframes_{\blam}(\br) = \Phi^{-1}(\mathcal{O}_{\blam}) \subset\pframes(\br)$ is connected.
\end{prop}

Note that, if we think of fusion frames in the usual way as a collection of subspaces or, equivalently, orthogonal projectors, this already gives an affirmative answer to \Cref{q:fusion frame homotopy} in the case $\K = \C$.

In the more general operator-valued frame case, the question is how $\pframes_{\blam}(\br)$ relates to $\coframes_{\blam}(\br)$, which is the space we want to show is connected. As previously discussed, while $A_i: \C^d \to \C^{k_i}$ uniquely determines $P_i = A_i^\ast A_i$, the operator $A_i$ cannot be uniquely determined from $P_i$. Indeed, if $U \in \unitary(k_i)$, then 
\[
	(UA_i)^\ast (UA_i) = A_i^\ast U^\ast U A_i = A_i^\ast A_i = P_i,
\]
so composing with a unitary transformation of the codomain leaves $P_i$ invariant. \Cref{prop:indeterminacy} says that this is the only indeterminacy, and hence the set of $P_i$ with given spectrum is precisely the collection of cosets of the (left) unitary action on the set of $A_i$ with given singular values:
\[
	\{P_i: \C^d \to \C^d \,|\, \spec{P_i} = \br_i\} \approx \{A_i:\C^d \to \C^{k_i} \,|\, \spec(A^\ast A) = \br_i\}/\unitary(k_i).
\]

In turn, this implies that
\begin{equation}\label{eq:pframes and oframes}
	\pframes(\br) \approx \oframes(\br)/(\unitary(k_1) \times \dots \times \unitary(k_N)),
\end{equation}
and hence that
\[
	\pframes_{\blam}(\br) \approx \oframes_{\blam}(\br)/(\unitary(k_1) \times \dots \times \unitary(k_N)).
\]
Here, we are using the fact that the operations of taking a level set and taking a quotient commute with one another. That is, suppose that $X$ is a topological space with an action by a group $G$ and $\Psi:X \to Y$ is a $G$-invariant map, so that the induced quotient map $\widetilde{\Psi}:X/G \to Y$ taking an equivalence class $[x]$ to $\Psi(x)$ is well defined; then $\widetilde{\Psi}^{-1}(y) = \Psi^{-1}(y)/G$ for any $y \in Y$. Indeed, $[x] \in \widetilde{\Psi}^{-1}(y)$ if and only if $x \in \Psi^{-1}(y)$, or equivalently $[x] \in \Psi^{-1}(y)/G$. 

Therefore, the following lemma (with $X = \oframes_{\blam}(\br)$ and $G = \unitary(k_1) \times \dots \times \unitary(k_N)$) combined with \Cref{prop:connected operators} implies that $\oframes_{\blam}(\br)$ is connected. Since $\oframes_{\blam}(\br)$ is a real algebraic set in $\C^{k_1\times d} \times \dots \times \C^{k_N \times d}\approx \R^{2(k_1 + \dots + k_N)d}$, it is locally path-connected, so that connectivity implies path-connectivity, completing the proof of \Cref{thm:main}.

\begin{lem}\label{lem:connected quotient}
	Let $X$ be a topological space and let $G$ be a connected topological group acting continuously on $X$. If $X/G$ is connected, then $X$ is connected.
\end{lem}
The lemma follows from standard point-set topology arguments (see, e.g.,~\cite[Exercise~5.5]{Manetti:2015ep}), since connectedness of $G$ implies the fibers of the quotient map are connected.

\section{Tightening Fusion Frames} 
\label{sec:benedetto fickus}

We now specialize to fusion frames, but relax our assumption on the base field, so that $\K = \R$ or $\C$. Recall the definition of the fusion frame potential $\FFP: \fframes \to \R$:
	\[
		\FFP(\mbf{A}) := \left\|S_{\mbf{A}}\right\|^2.
	\]

One of Casazza and Fickus' first results about the fusion frame potential was a Welch-type lower bound:

\begin{prop}[{\cite[Proposition~1]{casazza_minimizing_2009}}]\label{prop:welch bound}
	Let $d,N,k_1, \dots , k_N$ be positive integers and define $n:= k_1 + \cdots + k_N$. Then for $\mbf{A} = (A_1, \dots , A_N) \in \fframes$,
	\[
		\FFP(\mbf{A}) \geq \frac{1}{d}\left(\sum_{i=1}^N k_i\right)^2 = \frac{n^2}{d}
	\]
	with equality if and only if $\mbf{A}$ is a tight fusion frame.
\end{prop}

Let $\Lambda = \frac{n}{d}$ and define $\bLam=(\Lambda, \dots , \Lambda)$. If they exist, the TFFs must have frame operator $\Lambda \Id_d$, so that $\fframes_{\bLam} \subset \fframes$ is exactly the collection of TFFs. Hence, \Cref{prop:welch bound} says that if this collection of TFFs is nonempty then it is exactly the set of global minimizers of $\FFP$ in $\fframes$.

The hard work of proving \Cref{thm:descent} is in proving the $\K = \C$ case. The real case then follows immediately since $\rfframes \subset \cfframes$ is invariant under the gradient flow of $\FFP$. We record this observation in the following proposition:

\begin{prop}\label{prop:complex implies real}
    For given $d$ and $\bk$, if \Cref{thm:descent} is true for $\cfframes$, then it is also true for $\rfframes$.
\end{prop}

Except where explicitly pointed out below, we will assume $\K=\C$ in what follows. The strategy for proving the complex case of \Cref{thm:descent} is to show that property $\mathscr{S}$ (\Cref{def:property S}) satisfies the following conditions:
\begin{enumerate}[(i)]
	\item \label{it:second good property} gradient flow (and its limit) preserves $\mathscr{S}$, but
	\item \label{it:last good property} no non-minimizing critical point of the operator-valued frame potential satisfies $\mathscr{S}$.
\end{enumerate}

Our argument has roots in Mumford's \emph{geometric invariant theory (GIT)}~\cite{mumford_geometric_1994} (see~\cite{thomas_notes_2005} for a nice introduction), which we introduce in some generality before specializing.

Let $G$ be a reductive algebraic group that acts linearly on a finite-dimensional complex vector space $V$. For example, $G$ might be $GL(V)$ or $SL(V)$. A nonzero vector $v \in V$ is \emph{unstable} under the action of $G$ if the closure $\overline{G \cdot v}$ of the $G$-orbit of $v$ contains the origin; otherwise $v$ is \emph{semi-stable}. Notice that the unstable points are precisely those in the vanishing locus of every $G$-invariant homogeneous polynomial on $V$, and hence the semi-stable points are those on which some $G$-invariant homogeneous polynomial does not vanish.

As one might expect, semi-stability is a feature of the orbit of $v$: either the entire orbit consists of semi-stable points, or the entire orbit consists of unstable points.

\begin{prop}[{see, e.g.,~\cite[Proposition~6]{mixon_three_2021}}]\label{prop:semistable orbits}
	Given a nonzero $v \in V$ that is semi-stable, every point in $\overline{G \cdot v}$ is also semi-stable.
\end{prop}

\subsection{$V$ and the $\SL(d)$ Action}
\label{sub:action}

In our application of GIT, we will have $G = \SL(d)$. To determine the appropriate vector space $V$, we first recall the \emph{Pl\"ucker embedding} $\Gr{k}{d} \to \mathbb{P}(\bigwedge^k \C^d)$, defined on the Grassmannian $\Gr{k}{d}$ of $k$-dimensional linear subspaces of $\C^d$. We can represent a $k$-plane by any basis $v_1, \dots , v_k$ for it. Then the Pl\"ucker embedding is defined to be the projectivization of the map $\tau: (\C^d)^k \to \bigwedge^k \C^d$ defined by
\[
	\tau(v_1, \dots , v_k) := v_1 \wedge \dots \wedge v_k.
\]
When $(u_1, \dots , u_k)$ and $(v_1, \dots , v_k)$ span the same $k$-dimensional subspace, then $\tau(u_1, \dots , u_k) = \det(h) \tau(v_1, \dots , v_k)$ where $h \in \GL(k)$ is the change-of-basis matrix, so both map to the same point in projective space, and the Pl\"ucker embedding is well-defined on the Grassmannian. Of course, the standard action of $\SL(d)$ on $\C^d$ induces an action on $\bigwedge^k \C^d$ by
\[
	g \cdot (v_1 \wedge \dots \wedge v_k) := (g v_1) \wedge \dots \wedge (gv_k)
\]
and extending linearly.

How do we get from fusion frames to Grassmannians, and hence to such a representation of $\SL(d)$?

For $\mbf{A}  \in \cfframes$, each $P_i = A_i^\ast A_i$ is a rank-$k_i$ orthogonal projector and the rows $a_{i1}, \dots , a_{ik_i}$ of $A_i$ give an orthonormal basis for the $k_i$-dimensional subspace which is the image of $P_i$. Moreover, $\mathcal{O}_{\br_i}$ is the collection of all rank-$k_i$ orthogonal projectors, which is symplectomorphic to the Grassmannian $\Gr{k_i}{d}$. 

Define $\tau_i: A_i \mapsto a_{i1}^\ast \wedge \dots \wedge a_{ik_i}^\ast$, the projectivization of which is exactly the Pl\"ucker embedding of $\Gr{k_i}{d}$, and $\tau_i$ is equivariant with respect to the right $\SL(d)$ action $g \cdot A_i := A_i g^\ast$ on the domain and the $\SL(d)$ action described above on the codomain:
\[
	\tau(g \cdot A_i) = \tau(A_i g^\ast) = (g a_{i1}^\ast) \wedge \dots \wedge (g a_{ik_i}^\ast) = g \cdot \tau(A_i)
\]
for any $g \in \SL(d)$. 

Next, define $\tau: \cfframes \to \bigwedge^{k_1} \C^d \otimes \dots \otimes \bigwedge^{k_N} \C^d$ by $\tau := \tau_1 \otimes \dots \otimes \tau_N$, so that
\[
	\tau(\mbf{A}) = \tau_1(A_1) \otimes \dots \otimes \tau_N(A_N) = \left(a_{11}^\ast \wedge \dots \wedge a_{1k_1}^\ast\right) \otimes \dots \otimes \left(a_{N1}^\ast \wedge \dots \wedge a_{Nk_N}^\ast\right).
\]
In other words, the projectivization of $\tau$ is the Segre embedding of the product of Pl\"ucker embeddings of the individual factors.

Finally, then, our vector space $V = \bigwedge^{k_1} \C^d \otimes \dots \otimes \bigwedge^{k_N} \C^d$, on which $G = \SL(d)$ acts by
\[
	g \cdot \left((v_{11} \wedge \dots \wedge v_{1k_1}) \otimes \dots \otimes (v_{N1} \wedge \dots \wedge v_{Nk_N})\right) := \left( (g v_{11}) \wedge \dots \wedge (g_{1k_1})) \otimes \dots \otimes ((gv_{N1}) \wedge \dots \wedge (gv_{Nk_N}))\right)
\]
and extending linearly.

The point of defining property $\mathscr{S}$ as we have is the following theorem of Mumford (stated originally in terms of Grassmannians):

\begin{thm}[{Mumford~\cite[Proposition~4.3]{mumford_geometric_1994}; see also~\cite{mumford_projective_1963} and~\cite[\S16.3]{Kirwan:1984jt}}]\label{thm:semi-stability equals S}
	$\mbf{A}$ has property $\mathscr{S}$ if and only if $\tau(\mbf{A})$ is semi-stable with respect to the $\SL(d)$ action.
\end{thm}

As pointed out just before \Cref{prop:semistable orbits}, the semi-stable points in $V$ are those on which some $G$-invariant homogeneous polynomial does not vanish. Hence, \Cref{thm:semi-stability equals S} implies that if $\mbf{A} \in \cfframes$ has property $\mathscr{S}$, then there is some $G$-invariant homogeneous polynomial which does not vanish at $\tau(\mbf{A})$. Since the coordinates of $\tau(\mbf{A})$ are precisely the determinants of all the $k_i \times k_i$ minors of the $A_i$, and since these determinants are themselves polynomials in the entries of $\mbf{A}$, this means that there is some polynomial expression in the coordinates of $\mbf{A}$ which does not vanish. Therefore, the collection of $\mbf{A}$ with property $\mathscr{S}$ is Zariski-open in the smooth, connected, real algebraic variety $\cfframes$, and hence is either empty or dense (see, e.g.,~\cite[Proposition~5.11]{cahill2017connectivity}). Moreover, the same reasoning applies in $\rfframes$. Therefore, we have:

\begin{prop}\label{prop:dense}
	Let $\K = \R$ or $\C$. When it is non-empty, the collection of all fusion frames in $\fframes$ with property $\mathscr{S}$ is dense. 
\end{prop}

We also take this opportunity to show that TFFs always have property $\mathscr{S}$.

\begin{prop}\label{prop:tight implies S}
	Let $\K = \R$ or $\C$ and suppose $\mbf{A} \in \fframes$ is a TFF. Then $\mbf{A}$ has property $\mathscr{S}$.
\end{prop}

\begin{proof}
	Since $\mbf{A}$ is tight, its frame operator $S_{\mbf{A}} = \frac{n}{d}\Id_d$. Let $\mathcal{Q}\subset \K^d$ be a proper subspace and let $P_{\mathcal{Q}}$ be orthogonal projection onto $\mathcal{Q}$. 
	
	For each $i=1,\dots , N$, any nonzero vector in $\mathcal{Q} \cap \mathcal{S}_i$ is fixed by the product $P_{\mathcal{Q}}P_i$, and hence $\tr(P_{\mathcal{Q}}P_i) \geq \dim(\mathcal{Q} \cap \mathcal{S}_i)$, since all eigenvalues of $P_{\mathcal{Q}}P_i$ are real and non-negative. Therefore, since $P_{\mathcal{Q}}S_{\mbf{A}} = \frac{n}{d}P_{\mathcal{Q}}$,
	\[
		\frac{n}{d} \dim(\mathcal{Q}) = \tr(P_{\mathcal{Q}}S_{\mbf{A}}) = \sum_{i=1}^N \tr(P_{\mathcal{Q}}P_i) \geq \sum_{i=1}^N \dim(\mathcal{Q} \cap \mathcal{S}_i),
	\]
	so $\mbf{A}$ has property $\mathscr{S}$.
\end{proof}

Combining \Cref{prop:dense,prop:tight implies S} yields the following immediate corollary:

\begin{cor}\label{cor:dense}
	Whenever there are TFFs in $\fframes$, the fusion frames with property $\mathscr{S}$ are dense in $\fframes$.
\end{cor}

\subsection{Property $\mathscr{S}$ Satisfies \ref{it:second good property} and \ref{it:last good property}}
\label{sub:good properties}

\subsubsection{Property $\mathscr{S}$ Satisfies \ref{it:second good property}}\label{subsub:property ii}

The goal in this subsection is to show that the gradient flow of $\FFP$ preserves property $\mathscr{S}$. 

Notice that, if $\mbf{A} = (A_1, \dots, A_N)$ is a fusion frame, then the rows of each $A_i$ form an orthonormal set, and hence each $a_{i1}^\ast \wedge \dots \wedge a_{ik_i}^\ast \in \bigwedge^{k_i} \C^d$ is a unit vector with respect to the standard inner product on $\bigwedge^{k_i} \C^d$. In turn, this implies that $\tau(\mbf{A}) \in \left( \bigwedge^{k_1} \C^d \right) \otimes \dots \otimes \left(\bigwedge^{k_N} \C^d\right)$ is also a unit vector. In other words:
\begin{lem}\label{lem:bounded away from zero}
	$\tau\left(\cfframes\right)$ is contained in the unit sphere, and in particular is bounded away from the origin.
\end{lem}

Now we compute the gradient of $\OFP$, first by computing the extrinsic gradient of its extension to the entire vector space containing $\cfframes$:

\begin{lem}\label{lem:extrinsic gradient}
	Define $\EFP: \C^{k_1 \times d} \times \dots \times \C^{k_N \times d} \to \R$ to be the extension of $\OFP$ to all of $\C^{k_1 \times d} \times \dots \times \C^{k_N \times d}$ given by $\EFP(A_1, \dots , A_N) := \left\|\sum_{i=1}^N A_i^\ast A_i \right\|^2$. Its gradient is
	\begin{equation}\label{eqn:EFP_gradient}
		\nabla \EFP(\mbf{A}) =  (4 A_1 S_{\mbf{A}}, \ldots, 4 A_N S_{\mbf{A}}).
	\end{equation}
\end{lem}

\begin{proof}
    Let $\mbf{B} = (B_1,\ldots,B_N) \in T_{\mbf{A}} (\C^{k_1 \times d} \times \dots \times \C^{k_N \times d}) \approx \C^{k_1 \times d} \times \dots \times \C^{k_N \times d}$ and consider the directional derivative of $\EFP$ at $\mbf{A}$ in the direction $\mbf{B}$. In the following, we slightly abuse notation and generically use $\langle \cdot, \cdot \rangle$ for the Frobenius inner product on matrix spaces of various dimensions. We have 
    \begin{align}
        \left. \frac{d}{d\epsilon} \right|_{\epsilon = 0} \EFP(\mbf{A} + \epsilon \mbf{B}) &= \left. \frac{d}{d\epsilon} \right|_{\epsilon = 0} \left\| \sum_{i=1}^N (A_i + \epsilon B_i)^\ast (A_i + \epsilon B_i) \right\|^2 \nonumber \\
        &= \left. \frac{d}{d\epsilon} \right|_{\epsilon = 0} \left<\sum_{i=1}^N (A_i^\ast A_i + \epsilon (A_i^\ast B_i + B_i^\ast A_i) + \epsilon^2 B_i^\ast B_i),\sum_{i=1}^N (A_i^\ast A_i + \epsilon (A_i^\ast B_i + B_i^\ast A_i) + \epsilon^2 B_i^\ast B_i)\right> \nonumber \\
        &= 2 \mathrm{Re} \left<\sum_{i=1}^N A_i^\ast A_i, \sum_{i=1}^N (A_i^\ast B_i + B_i^\ast A_i) \right> \nonumber \\
        &= 4 \mathrm{Re} \sum_{i=1}^N \left< S_{\mbf{A}}, A_i^\ast B_i \right> \label{eqn:gradient_computation_1} \\
        &= \mathrm{Re} \sum_{i=1}^N \left<4 A_i S_{\mbf{A}}, B_i\right> = \mathrm{Re} \left<(4A_1 S_{\mbf{A}},\ldots,4 A_N S_{\mbf{A}}), (B_1,\ldots,B_N) \right> \label{eqn:gradient_computation_2}.
    \end{align}
The equality \eqref{eqn:gradient_computation_1} makes the replacement $S_{\mbf{A}} = \sum_{i=1}^N A_i^\ast A_i$, uses linearity in the second coordinate to move the summation out of the inner product and uses properties of the (real part of the) inner product to equate $\mathrm{Re}\langle \cdot, A_i^\ast B_i + B_i^\ast A_i\rangle = 2 \mathrm{Re}\langle \cdot, A_i^\ast B_i \rangle$. The quantity  \eqref{eqn:gradient_computation_2} is the (real part of the) standard inner product for $\C^{k_1 \times d} \times \dots \times \C^{k_N \times d}$ applied to $(4A_1 S_{\mbf{A}},\ldots,4 A_N S_{\mbf{A}})$ (the claimed formula for the gradient) and $\mbf{B}$. This implies \eqref{eqn:EFP_gradient}. 
\end{proof}

And now the intrinsic gradient:

\begin{lem}\label{lem:gradient}
	The Riemannian gradient of $\OFP: \cfframes \to \R$, is
	\[
		\grad \OFP(\mbf{A}) = \left(4(A_1 S_{\mbf{A}} - (A_1 S_{\mbf{A}} A_1^\ast) A_1), \ldots, 4(A_1 S_{\mbf{A}} - (A_1 S_{\mbf{A}} A_1^\ast) A_1)\right).
	\]
\end{lem}

\begin{proof}
	The Riemannian gradient $\grad \OFP(\mbf{A})$ is  the projection of the extrinsic gradient $\nabla \EFP(\mbf{A})$ onto the tangent space to $\cfframes$ at $\mbf{A}$. This means that on the $i$th block we need to project $4A_i S_{\mbf{A}}$ onto the orthogonal complement of the row span of $A_i$. This orthogonal projection is accomplished by right-multiplying by $(\Id_d - A_i^\ast A_i)$, so the projection is $4A_i S_{\mbf{A}}(\Id_d - A_i^\ast A_i) = 4\left(A_i S_{\mbf{A}} - \left(A_i S_{\mbf{A}} A_i^\ast\right)A_i\right)$. The result follows.
\end{proof}

\begin{prop}\label{prop:property ii}
	Suppose $\mbf{A}_0 \in \cfframes$ has property $\mathscr{S}$ and that $\Gamma: \fframes \times [0,\infty) \to \fframes$ is the gradient flow defined in \Cref{thm:benedetto fickus}. Then $\mbf{A}_\infty := \lim_{t\to\infty} \Gamma(\mbf{A}_0,t)$ has property $\mathscr{S}$.
\end{prop}

\begin{proof}
    Using \Cref{lem:gradient}, we have that the $i$th block of $\grad \OFP(\mbf{A})$ is 
	\begin{equation}\label{eqn:gradient_to_exponential}
		4\left(A_i S_{\mbf{A}} - \left(A_i S_{\mbf{A}} A_i^\ast\right)A_i\right) =  \left. \frac{d}{d \epsilon} \right|_{\epsilon=0} 4\exp(-\epsilon A_i S_{\mbf{A}} A_i^\ast) A_i \exp(\epsilon S_{\mbf{A}}).
	\end{equation}
	Exponentiating a matrix always yields an invertible matrix, so \eqref{eqn:gradient_to_exponential} tells us that $\grad \OFP(\mbf{A})$ is tangent to the orbit $\left(\GL(d) \times \prod_{i=1}^N \GL(k_i)\right) \cdot \mbf{A}$, where $(g,(h_1,\dots , h_N)) \in \GL(d) \times \prod_{i=1}^N \GL(k_i)$ acts on $\prod_{i=1}^N \C^{k_i \times d}$ by 
	\[
		(g,(h_1,\dots , h_N)) \cdot \mbf{A} = (h_1 A_1 g^\ast, \dots , h_N A_n g^\ast).
	\]
	
	For $(g,(h_1,\dots , h_N)) \in \GL(d) \times \prod_{i=1}^N \GL(k_i)$, we normalize $g$ to get something in $\SL(d)$ without changing the action by moving a scalar to the other factor: $\left((\det g)^{-1/d}g, \left((\overline{\det g})^{-1/d} h_1, \dots , (\overline{\det g})^{-1/d}, h_N\right)\right) \in \SL(d)  \times \prod_{i=1}^N \GL(k_i)$ and
	\[
		(g,(h_1,\dots , h_N)) \cdot \mbf{A} = \left((\det g)^{-1/d}g, \left((\overline{\det g})^{-1/d} h_1, \dots , (\overline{\det g})^{-1/d}, h_N\right)\right) \cdot \mbf{A},
	\]
	so $\grad \OFP(\mbf{A})$ is actually in the tangent space to $\left(\SL(d) \times \prod_{i=1}^N \GL(k_i)\right) \cdot \mbf{A}$ at $\mbf{A}$. Therefore,
	\[
		\Gamma(\mbf{A}_0,t) \in \left(\SL(d) \times \prod_{i=1}^N \GL(k_i)\right) \cdot \mbf{A}_0 \quad \text{ for all } t \geq 0.
	\]
	
	If $(g,(h_1,\dots , h_N) \in \SL(d) \times \prod_{i=1}^N \GL(k_i)$, then
	\begin{align*}
		\tau((g,(h_1,\dots , h_N)) \cdot \mbf{A}) & = \left( \left( g a_{11}^\ast h_1^\ast\right) \wedge \dots \wedge \left( g a_{1k_1}^\ast h_1^\ast\right)\right) \otimes \dots \otimes \left( \left( g a_{N1}^\ast h_N^\ast\right) \wedge \dots \wedge \left( g a_{Nk_N}^\ast h_N^\ast\right)\right)  \\
		& = \left( \det(h_1^\ast) \left( g a_{11}^\ast\right) \wedge \dots \wedge \left( g a_{1k_1}^\ast\right)\right) \otimes \dots \otimes \left( \det(h_N^\ast) \left( g a_{N1}^\ast\right) \wedge \dots \wedge \left( g a_{Nk_N}^\ast\right)\right) \\
		& = \prod_{i=1}^N \det(h_i^\ast)\left[ \left(\left( g a_{11}^\ast\right) \wedge \dots \wedge \left( g a_{1k_1}^\ast\right)\right) \otimes \dots \otimes \left( \left( g a_{N1}^\ast\right) \wedge \dots \wedge \left( g a_{Nk_N}^\ast\right)\right)\right] \in \left( \SL(d) \times \C^\times\right)\cdot \tau(\mbf{A}),
	\end{align*}
	where $(g,a) \in \SL(d) \times \C^\times$ acts on $\left( \bigwedge^{k_1} \C^d \right) \otimes \dots \otimes \left(\bigwedge^{k_N} \C^d\right)$ by
	\[
		(g,a) \cdot \left[\left( v_{11} \wedge \dots \wedge v_{1k_1}\right) \otimes \dots \otimes \left( v_{N1} \wedge \dots \wedge v_{Nk_N}\right)\right] := a \left[ \left( \left( g v_{11}\right) \wedge \dots \wedge \left( g v_{1k_1}\right) \right) \otimes \dots \otimes \left( \left( g v_{N1} \right) \wedge \dots \wedge \left(g v_{Nk_N}\right)\right)\right].
	\]
	
	This implies that $\tau(\Gamma(\mbf{A}_0,t)) \in \left( \SL(d) \times \C^\times\right)\cdot \tau(\mbf{A}_0)$ for all $t \geq 0$. Since $\tau(\Gamma(\mbf{A}_0,t))$ is a unit vector for all $t$ by \Cref{lem:bounded away from zero}, so is the limit $\tau(\mbf{A}_\infty)$. 
	
	Since everything is bounded away from the origin and since rescaling a vector by a nonzero scalar does not affect its semistability with respect to the $\SL(d)$-action, \Cref{prop:semistable orbits} implies that the entire gradient flow line, including $\tau(\mbf{A}_\infty)$, is semi-stable, and hence $\mbf{A}_\infty$ has property $\mathscr{S}$ by \Cref{thm:semi-stability equals S}.
\end{proof}

\subsubsection{Property $\mathscr{S}$ Satisfies \ref{it:last good property}}\label{subsub:property iii}

Finally, we need to show that critical points which are not global minima do not satisfy property $\mathscr{S}$. We do this by showing that, if $\mbf{A}$ is a non-minimizing critical point, then $\tau(\mbf{A})$ is not semi-stable with respect to the $\SL(d)$ action. Semi-stability is defined in terms of the full group orbit, but this is typically much too big to be tractable. Instead, it is preferable to work with one-parameter subgroups, which remarkably turn out to be sufficient.

We briefly return to discussing a general reductive group $G$ acting linearly on a vector space $V$. A \emph{one-parameter subgroup} of $G$ is a homomorphism of algebraic groups $\lambda: \C^\times \to G$. Any such homomorphism induces a decomposition $V = \bigoplus_{i \in I} V_i$ and integer weights $w: I \to \Z$ so that, for every $i \in I$, $v \in V_i$, and $t \in \C^\times$,
\[
	\lambda(t) \cdot v = t^{w(i)}v.
\]
It follows immediately from the definition that a nonzero vector $v \in V$ is unstable under the action of $G$ if there exists a one-parameter subgroup $\lambda$ so that
\[
	\lim_{t\to 0} \lambda(t) \cdot v = 0.
\]
Much less obvious is that the converse holds:

\begin{thm}[Hilbert--Mumford criterion~\cite{mumford_geometric_1994,hilbert_ueber_1893}]\label{thm:hilbert-mumford}
	$v \in V\backslash\{0\}$ is unstable under the action of $G$ if and only if there exists a one-parameter subgroup $\lambda$ of $G$ so that
	\[
		\lim_{t \to 0} \lambda(t) \cdot v = 0.
	\]
\end{thm}

This will be our key tool in proving that property $\mathscr{S}$ satisfies \ref{it:last good property}.

\begin{prop}\label{prop:property iii}
	Suppose $\mbf{A} \in \cfframes$ is a critical point of $\FFP$ which is not tight. Then $\mbf{A}$ does not have property $\mathscr{S}$. In particular, if $\mbf{A}$ is a critical point which is not a global minimum, then it does not have property $\mathscr{S}$.
\end{prop}

\begin{proof}
	Let $\mbf{A} \in \cfframes$ be a critical point of $\OFP$. Then $\grad \OFP(\mbf{A}) = 0$; by \Cref{lem:gradient}, this implies that, for each $i = 1, \dots , N$,
	\[
		0 = A_i S_{\mbf{A}} - A_i S_{\mbf{A}} A_i^\ast A_i = A_i S_{\mbf{A}}( \Id_d - A_i^\ast A_i ).
	\]
	The operator $\Id_d - A_i^\ast A_i$ is orthogonal projection onto the orthogonal complement of $\row(A_i)$, the row space of $A_i$ (since $\mbf{A}$ is a fusion frame, the rows of $A_i$ are orthonormal). The above equation then says that the rows of $A_i S_{\mbf{A}}$ lie in $\row(A_i)$. In other words, $\row(A_i)$ is an invariant subspace for the frame operator $S_{\mbf{A}}$, and hence has an orthonormal basis of eigenvectors of $\left.S_{\mbf{A}}\right|_{\row(A_i)}$, so there exists $U_i \in \unitary(k_i)$ so that the rows of $\widetilde{A}_i := U_i A_i$ are eigenvectors of $\left.S_{\mbf{A}}\right|_{\row(A_i)}$, and hence also of $S_{\mbf{A}} = S_{\widetilde{\mbf{A}}}$. So far, this is not new: the conclusion of the previous sentence is exactly Casazza and Fickus' characterization of the critical points of the fusion frame potential~\cite[Theorem~4]{casazza_minimizing_2009}.
	
	If $\mbf{A}$ is not tight, then the frame operator $S_{\mbf{A}}$ has at least two distinct eigenvalues. Let $\lambda$ be the largest eigenvalue, with corresponding eigenspace $E_\lambda$ of dimension $\ell$ and orthogonal complement $E_\lambda^\bot$ of dimension $d-\ell$. Since the average of the eigenvalues of $S_{\mbf{A}}$ is
	\[
		\frac{1}{d} \tr(S_{\mbf{A}}) = \frac{1}{d} \tr\left(\sum_{i=1}^N A_i^\ast A_i \right) = \frac{1}{d}\sum_{i=1}^N \tr \left(A_i^\ast A_i \right) = \frac{1}{d}\sum_{i=1}^N k_i = \frac{n}{d}
	\]
	and the eigenvalues aren't all equal, we know that the largest eigenvalue $\lambda > \frac{n}{d}$.
	
	 Up to conjugating $S_{\mbf{A}}$ by $U \in \unitary(d)$ (corresponding to right-multiplying each $A_i$ by $U^\ast$), we can make the simplifying assumption that the frame operator is diagonal: $S_{\mbf{A}} = \begin{bmatrix} \lambda \Id_\ell & 0 \\ 0 & S' \end{bmatrix}$, where $S'$ is a diagonal (but not necessarily scalar) matrix. Hence, $E_\lambda = \spa\{e_1, \dots , e_\ell\}$ and $E_\lambda^\bot = \spa\{e_{\ell+1} , \dots , e_d\}$. 
	
	If, for each $i=1,\dots , N$, $\tilde{a}_{i1},\dots, \tilde{a}_{ik_i}$ are the rows of $\widetilde{A}_i$, then, since each $\tilde{a}_{ij}$ is an eigenvector of $S_{\mbf{A}}$ and distinct eigenspaces are orthogonal, the $\tilde{a}_{ij}$ split into two perpendicular groups: those in $E_\lambda$ and those in $E_\lambda^\bot$. Let $\tilde{b}_1, \dots , \tilde{b}_m$ be the collection contained in $E_\lambda$. Then
	\[
		\sum_{i=1}^m \tilde{b}_i^\ast \tilde{b}_i = \begin{bmatrix} \lambda \Id_\ell & 0 \\ 0 & 0 \end{bmatrix}
	\]
	so that $\tilde{b}_1^\ast, \dots , \tilde{b}_m^\ast$ are a $\lambda$-tight frame for $E_\lambda$. Since the $\tilde{b}_i$ are unit vectors,
	\[
		\ell\lambda = \tr \begin{bmatrix} \lambda \Id_\ell & 0 \\ 0 & 0 \end{bmatrix} = \tr(\sum_{i=1}^m \tilde{b}_1^\ast \tilde{b}_m) = \tr(\sum_{i=1}^m \tilde{b}_i\tilde{b}_i^\ast) = m,
	\]
	it follows that $\lambda = \frac{m}{\ell}$ and hence that $\frac{m}{\ell} > \frac{n}{d}$; equivalently, $md-n\ell > 0$.
	
	We're now ready to show that $\mbf{A}$ does not have property $\mathscr{S}$. To see this, consider the 1-parameter subgroup $\lambda: \C^\times \to \SL(d)$ given as a block matrix by
	\[
		\lambda(t) = \begin{bmatrix} t^{d-\ell}\Id_\ell & 0 \\ 0 & t^{-\ell}\Id_{d-\ell} \end{bmatrix}.
	\]
	Since $\tilde{a}_{i1}^\ast \wedge \dots \wedge \tilde{a}_{ik_i}^\ast = \det(U_i^\ast) a_{i1}^\ast \wedge \dots \wedge a_{ik_i}^\ast$, it follows that $\tau(\mbf{A}) = \rho \tau(\widetilde{\mbf{A}})$ for unimodular $\rho = \prod_{i=1}^N \det(U_i)$, and hence
	\begin{multline*}
		\lambda(t) \cdot \tau(\mbf{A}) = \rho\lambda(t) \cdot \tau(\widetilde{\mbf{A}}) = \rho\left(\left(\lambda(t) \tilde{a}_{11}^\ast\right) \wedge \dots \wedge \left(\lambda(t) \tilde{a}_{1k_1}^\ast \right)\right) \otimes \dots \otimes \left(\left(\lambda(t) \tilde{a}_{N1}^\ast\right) \wedge \dots \wedge \left(\lambda(t) \tilde{a}_{Nk_N}^\ast \right)\right) \\
		= t^{m(d-\ell)-(n-m)\ell}\rho\left( \tilde{a}_{11}^\ast \wedge \dots \wedge \tilde{a}_{1k_1}^\ast \right) \otimes \dots \otimes \left( \tilde{a}_{N1}^\ast \wedge \dots \wedge  \tilde{a}_{Nk_N}^\ast \right) = t^{md-n\ell} \rho\tau(\widetilde{\mbf{A}}),
	\end{multline*}
	which goes to zero as $t\to 0$ since $md-n\ell > 0$.
	
	Therefore, $\tau(\mbf{A})$ is not semi-stable, and hence, by \Cref{thm:semi-stability equals S}, $\mbf{A}$ does not have property $\mathscr{S}$. 
	
	Finally, if $\mbf{A}$ is not a global minimum, then it cannot be tight by \Cref{prop:welch bound}, so we see that non-minimum critical points cannot have property $\mathscr{S}$.
\end{proof}

If $d, N$, and $\bk$ are such that $\fframes$ contains no TFFs (see \Cref{sec:admissibility} for conditions on when this occurs), then any global minimum $\mbf{A}$ of $\FFP$ cannot be tight, so \Cref{prop:property iii} shows that $\mbf{A}$ cannot have property $\mathscr{S}$. By the contrapositive of \Cref{prop:property ii}, then, nothing in $\fframes$ that flows to $\mbf{A}$ under the negative gradient flow of $\FFP$ can have property $\mathscr{S}$, either. Since this is true for all global minima, there is some open set containing the global minima which completely avoids the fusion frames with property $\mathscr{S}$. Therefore, the set of fusion frames in $\fframes$ with property $\mathscr{S}$ cannot be dense, and hence, by \Cref{prop:dense}, must be empty. In other words:

\begin{cor}\label{cor:tight iff S}
    $\fframes$ contains fusion frames with property $\mathscr{S}$ if and only if it contains TFFs.
\end{cor}

\subsubsection{Completing the Proof of \Cref{thm:descent}} 
\label{subsub:completing the proof}

We now have all the tools we need to prove that gradient descent limits to a global minimizer.

\begin{proof}[Proof of \Cref{thm:descent}]
	If $\mbf{A}_0 \in \cfframes$ has property $\mathscr{S}$, the limit $\mbf{A}_\infty := \lim_{t \to \infty} \Gamma(\mbf{A}_0, t)$ has property $\mathscr{S}$ by \Cref{prop:property ii}. Since $\mbf{A}_\infty$ is a limit point of the gradient flow, it must be a critical point of $\OFP$. Since it has property $\mathscr{S}$, \Cref{prop:property iii} implies that $\mbf{A}_\infty$ is a global minimizer of $\OFP$.
	
	This proves \Cref{thm:descent} when $\K = \C$. The real case then follows immediately by \Cref{prop:complex implies real}.
\end{proof}

\section{Discussion} 
\label{sec:discussion}

There are choices of $d$, $N$, and $\bk$ for which there are no fusion frames with property $\mathscr{S}$: for example, $d=N=3$ and $\bk = (1,1,2)$. Elements $(A_1, A_2, A_3) \in \mathcal{FF}^{\R^3,(1,1,2)}$ will determine two lines $\ell_1 = \row(A_1)$ and $\ell_2 = \row(A_2)$ and a plane $\mathcal{S} = \row(A_3)$. If $\mathcal{Q}$ is a plane containing $\ell_1$ and $\ell_2$, then it must intersect $\mathcal{S}$ at least in a line, so 
\[
	\frac{1}{\dim \mathcal{Q}}(\dim(\ell_1 \cap \mathcal{Q}) + \dim(\ell_2 \cap \mathcal{Q}) + \dim(\mathcal{S} \cap \mathcal{Q})) \geq \frac{3}{2} > \frac{4}{3} = \frac{\dim \ell_1 + \dim \ell_2 + \dim \mathcal{S}}{3}.
\]
Hence, nothing in $\mathcal{FF}^{\R^3,(1,1,2)}$ has property $\mathscr{S}$, so \Cref{thm:descent} tells us nothing. Moreover, by \Cref{prop:tight implies S}, there are no TFFs in $\mathcal{FF}^{\R^3,(1,1,2)}$. 

\begin{figure}[t]
	\centering
		\includegraphics[height=1.5in]{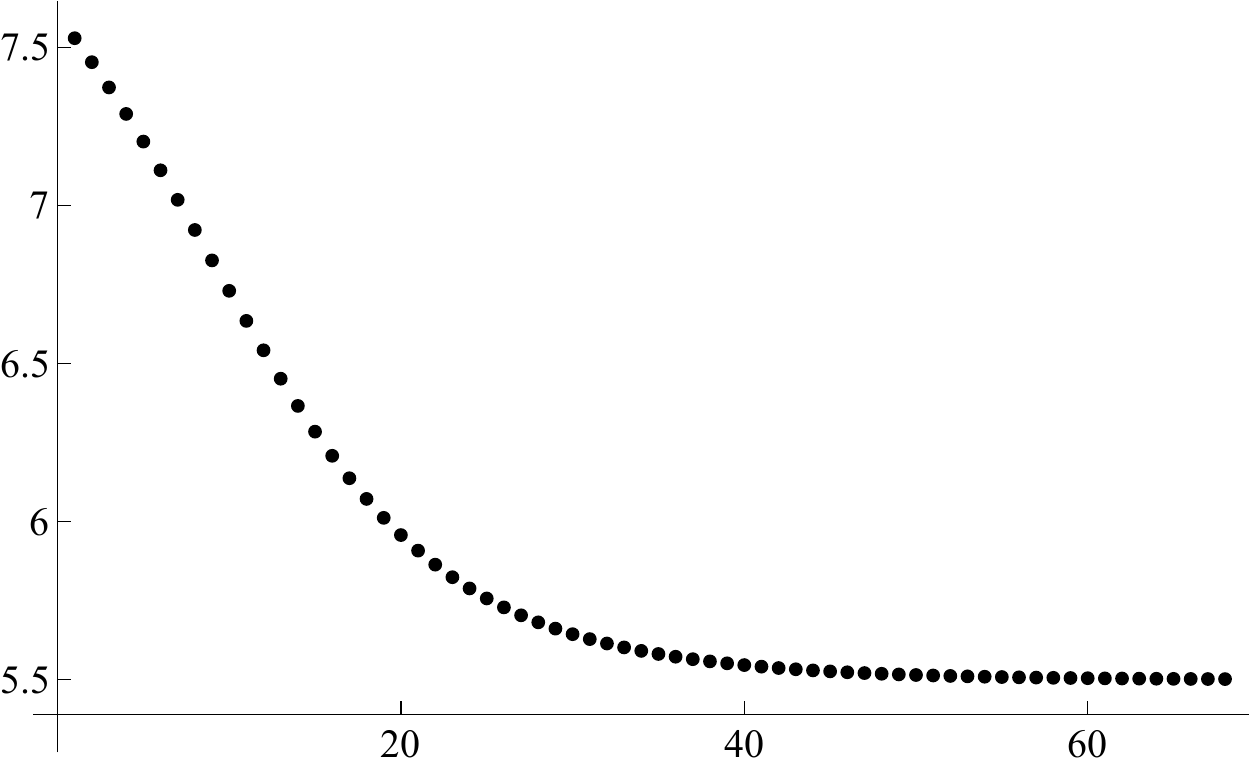}
		\qquad \qquad
		\includegraphics[height=1.5in]{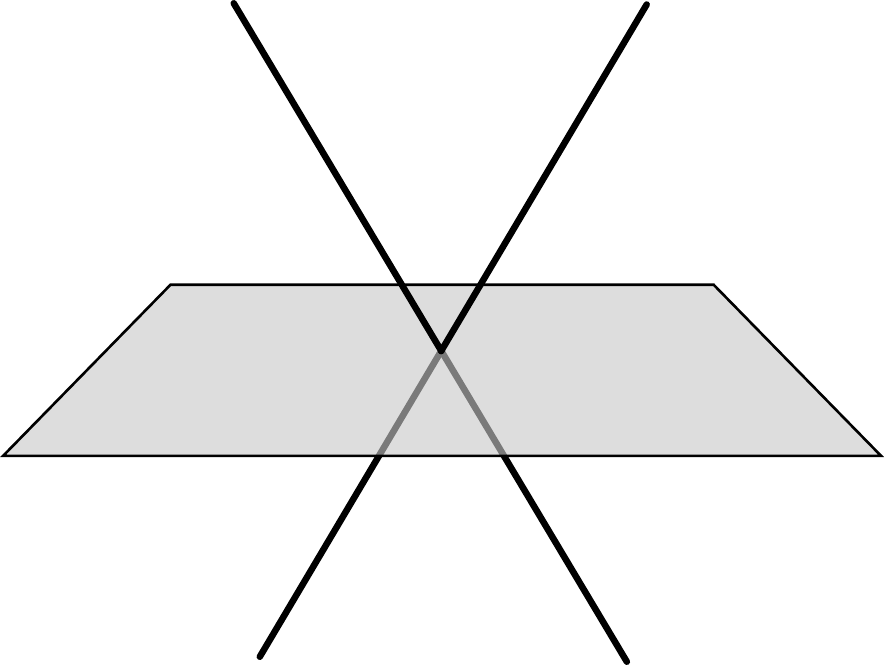}
	\caption{Left: Values of $\FFP$ at each step in a  simple gradient descent with fixed step sizes starting from a random point in $\mathcal{FF}^{\R^d,(1,1,2)}$. Right: The limiting fusion frame of this gradient descent which achieves the minimum possible value of $\frac{11}{2}$.}
	\label{fig:experiment}
\end{figure}

Nonetheless, running gradient descent from random starting points in $\mathcal{FF}^{\R^3,(1,1,2)}$ seems to always find minimizers of $\FFP$ in practice. \Cref{fig:experiment} shows the value of $\FFP$ rapidly decreasing to the global minimum value $\frac{11}{2}$, which is greater than the value $\frac{16}{3}$ that a TFF would have. As observed by Casazza and Fickus~\cite[p.~17]{casazza_minimizing_2009}, the minimum is achieved by fusion frames of the form shown on the right of \Cref{fig:experiment}, where $\ell_1$ and $\ell_2$ lie in a plane $\mathcal{Q}$ perpendicular to $\mathcal{S}$ and the lines $\ell_1$, $\ell_2$, and $\mathcal{Q} \cap \mathcal{S}$ correspond to a tight Mercedes--Benz frame for $\mathcal{Q}$.

This suggests that the fusion frame Benedetto--Fickus theorem may hold even for parameters where there are no fusion frames with property $\mathscr{S}$. We also expect our approach to proving \Cref{thm:descent} will extend to more general spaces $\oframes(\br)$ of operator-valued frames with fixed spectra of the $P_i$, though the details seem more complicated. Hence, we pose the following conjecture:

\begin{conj}
    Let $d,N,k_1, \dots k_N$ be positive integers and fix $\br$. Let $\operatorname{OFP}: \oframes(\br) \to \R$ be the obvious generalization of $\FFP$ to operator-valued frames. Then all local minima of $\OFP$ are global minima.
\end{conj}

We expect that, as in the case of classical frames~\cite{needham2022admissibility}, $\mathcal{OF}^{\mathh^d,\bk}_{\blam}(\br)$ is path-connected, where $\mathh$ is the skew-field of quaternions. However, $\mathcal{OF}^{\R^d,\bk}_{\blam}(\br)$ cannot always be connected. For example, translating a result of Kapovich and Millson~\cite[Theorem~1]{Kapovich:1995wg} to our setting and notation implies that $\mathcal{OF}^{\R^2,(1,1,1,1)}_{(5,5)}(3,3,3,1)$ is not connected. On the other hand, Cahill, Mixon, and Strawn~\cite{Cahill:2017gv} proved that the space $\mathcal{OF}^{\R^d,(1,\dots,1)}_{\left(\frac{N}{d}, \dots , \frac{N}{d}\right)}(1, \dots , 1)$ of real unit-norm tight frames is connected for all $d \geq 2$ and $N \geq d+2$, so there is some interesting characterization of when the $\mathcal{OF}^{\R^d,\bk}_{\blam}(\br)$ are connected still waiting to be discovered.

Cahill, Mixon, and Strawn's proof of the Frame Homotopy Theorem relied heavily on the use of \emph{eigensteps}, which are the eigenvalues of the partial sums of the $P_i$. While eigensteps can be similarly defined for fusion frames and even operator-valued frames, it is not clear whether they would be a useful tool for studying connectedness in the real case. Eigensteps give good coordinates for classical frame spaces because they are \emph{action coordinates}---that is, they are the coordinates of a momentum map for a Hamiltonian action of a half-dimensional torus~\cite{needham_toric_2021}. This means that not only is the image a convex polytope, but the fibers of the eigenstep map are reasonably simple and well-understood. For dimension-counting reasons it seems unlikely that eigensteps could give action coordinates for fusion frames or operator-valued frames, but it is desirable to find similarly useful coordinates in this more general setting.

\subsection*{Acknowledgments}

We are grateful to Dustin Mixon and Soledad Villar for discussions over the past year. Our joint chapter~\cite{mixon_three_2021} served as an important catalyst for this paper. This work was supported by grants from the National Science Foundation (DMS--2107808, Tom Needham; DMS--2107700, Clayton Shonkwiler).

\bibliography{needham_bibliography,shonkwiler-papers}

\end{document}